\documentclass{article}
\usepackage[utf8]{inputenc}
\usepackage{geometry}
\usepackage{amsmath,amsthm,amssymb,amsfonts}
\usepackage{comment}
\usepackage{wasysym}
\usepackage{color}
\usepackage[hidelinks]{hyperref}
\usepackage{enumerate}
\usepackage{authblk}
\usepackage{cleveref}
\usepackage{microtype}

\providecommand{\MR}[1]{}

\DeclareMathOperator{\LE}{LE}
\DeclareMathOperator{\C}{Cap}
\DeclareMathOperator{\Ca}{Cap}

\DeclareMathOperator{\SLE}{SLE}

\newcommand{\R}{\mathbb{R}}

\newcommand{\Z}{\mathbb{Z}}
\newcommand{\N}{\mathbb{N}}
\newcommand{\cF}{\mathcal{F}}
\renewcommand{\P}{\mathbb{P}}
\newcommand{\E}{\mathbb{E}}
\newcommand{\1}{\mathbf{1}}

\newcommand{\limn}{\lim_{n\rightarrow\infty}}
\newcommand{\limm}{\lim_{m\rightarrow\infty}}
\renewcommand{\l}{\ell}
\newcommand{\heta}{\hat\eta}

\newcommand{\Xh}{\hat X}
\newcommand{\Xhi}{\Xh_\infty}
\newcommand{\Xhip}{\Xhi^+}
\newcommand{\Cz}{\C_{\Z^3}}
\newcommand{\Cr}{\C_{\R^3}}

\newtheorem{theorem}{Theorem}

\newtheorem{lemma}[theorem]{Lemma}

\newtheorem{proposition}[theorem]{Proposition}

\theoremstyle{definition}
\newtheorem{definition}[theorem]{Definition}

\numberwithin{equation}{section}
\numberwithin{theorem}{section}

\begin{document}

\title{Capacity of loop-erased random walk}
\date{\vspace{-5ex}}
\author{Maarten Markering}
\affil{University of Cambridge}
\maketitle
\let\thefootnote\relax\footnotetext{Date: May 11, 2026}
\let\thefootnote\relax\footnotetext{Email: mjrm2@cam.ac.uk}
\let\thefootnote\relax
\let\thefootnote\relax
\newcommand{\thefootnote}{\arabic{footnote}}

\begin{abstract}
We study the capacity of loop-erased random walk (LERW) on $\mathbb{Z}^d$. For $d\geq4$, we prove a strong law of large numbers and give explicit expressions for the limit in terms of the non-intersection probabilities of a simple random walk and a two-sided LERW. Along the way, we show that four-dimensional LERW is ergodic. For $d=3$, we show that the scaling limit of the capacity of LERW is random. We show that the capacity of the first $n$ steps of LERW is of order $n^{1/\beta}$, with $\beta$ the growth exponent of three-dimensional LERW. We express the scaling limit of the capacity of LERW in terms of the capacity of Kozma's scaling limit of LERW. As a corollary, we obtain the scaling limit of the LERW in three dimensions when parametrized by its capacity.
\end{abstract}

\section{Introduction}\label{sec:intro}
We give an overview of the model, history, and results in Section~\ref{sec:background}. We state the laws of large numbers for dimensions $d\geq5$, $d=4$ and $d=3$ in Sections \ref{sec:d5res}, \ref{sec:d4res} and~\ref{sec:d3res} respectively. In Section~\ref{sec:sketch-lerw}, we give an overview of the proofs. Finally, in Section \ref{sec:outline-lerw}, we give an outline of the rest of the paper.

\subsection{Background}\label{sec:background}
In this paper, we study the \emph{capacity} of the \emph{loop-erased random walk} (LERW) on $\Z^d$ for $d\geq3$. LERW was introduced by Lawler in \cite{lawler80selfavoiding}. It has since become a central object in modern probability theory. This is partly due to its many connections to other models, chief among them being the uniform spanning tree. In that context, the \emph{capacity} of LERW has become an important object of study. The capacity of a set $A\subset\Z^d$ can be considered as a rescaled hitting probability of $A$ by a simple random walk (SRW) started far away. Through Wilson's algorithm, the capacity of a loop-erased random walk is related to the length and density of branches in the uniform spanning tree/forest. Using Hutchcroft's interlacement Aldous-Broder algorithm developed in \cite{hutchcroft2020universality}, the capacity of loop-erased random walk played an important role in determining the scaling of several exponents for the uniform spanning tree on $\Z^d$ for $d\geq5$ by Hutchcroft \cite{hutchcroft2020universality} and $d=4$ by Hutchcroft and Sousi \cite{hutchcroft2020logarithmic} and Hutchcroft and Halberstam \cite{halberstam22logarithmic}. In particular, a weak law of large numbers for the capacity of LERW was established in \cite{hutchcroft2020logarithmic}. A finite analogue of the capacity of LERW was also instrumental in establishing the scaling limit of the uniform spanning tree on the torus in dimensions $d\geq4$ \cite{archer24ghp, archer24dense, michaeli2021diameter, peres04scaling, schweinsberg2009loop}. In this paper, we prove several laws of large numbers for the capacity of LERW in dimensions $d\geq3$.

We now give a formal definition of the LERW and capacity. Let $\Omega$ be the set of finite, one-sided infinite and two-sided infinite nearest-neighbour paths in $\Z^d$. Let $\cF$ be the sigma-algebra on $\Omega$ generated by events depending only on a finite number of indices (which we call cylinder events). Throughout this paper, all events and random variables will be measurable with respect to $\cF$. Lastly, we define the left-shift operator $T\colon\Omega\to\Omega$, given by $(T\omega)(k)=\omega(k+1)-\omega(1)$.
\begin{definition}[Loop-erased random walk]\label{def:lerw}
Let $\omega\in\Omega$ be a one-sided nearest-neighbour, transient path in $\Z^d$. The loop-erasure $\LE(\omega)$ of $\omega$ is the path obtained by erasing loops from $\omega$ chronologically. To be precise, $\LE(\omega)_i=\omega_{\l_i}$, where the times $(\l_i)_{i\geq0}$ are defined inductively as $\l_0=0$ and $\l_{i+1}=1+\max\{k\colon\omega_k=\omega_{\l_i}\}$. When $S$ is a simple random walk on $\Z^d$ for $d\geq3$, we write
\begin{equation}\label{eq:def-lerw}
	\eta:=\LE(S[0,\infty)).
\end{equation}
This is well-defined, since simple random walk is almost surely transient for $d\geq3$.
\end{definition}
\begin{definition}[Capacity]\label{def:capacity}
Let $W$ be a simple random walk on $\Z^d$ for $d\geq 3$ and denote by $\P_y$ the law of $W$ with $W(0)=y\in\Z^d$. For a set $A\subset\Z^d$, the capacity of $A$ is defined as
\begin{equation}
    \Ca(A):=\lim_{\|y\|\rightarrow\infty}\frac{\P_y(W[0,\infty)\cap A\neq\varnothing)}{G(0,y)}=\sum_{a\in A}\P_a(W[1,\infty)\cap A=\varnothing),
\end{equation}
where $G$ is the SRW Green's function on $\Z^d$ and $\|y\|$ denotes the Euclidean norm of $y$.
\end{definition}
We refer to \cite{lawler2010random} for more background on the capacity, including a proof that the two expressions above are equal. Throughout this paper, we let $W,S$ denote simple random walks and denote their laws by $\P^W$ and $\P^S$. We drop the superscripts if the law is clear from the context. As in the previous two definitions, we will use $W$ in the context of a random walk hitting another set and we will use~$S$ in the context of a random walk being loop-erased.

Although motivated by applications to other models, we study the capacity of LERW in its own right in this paper. We establish a strong law of large numbers (SLLN) for $\Ca(\eta[0,n])$ in dimensions $d\geq4$. In three dimensions, we show that $\Ca(\eta[0,n])$ has a non-deterministic scaling limit. These complement earlier results on the capacity of SRW, which has a rich history. Already in 1968, Jain and Orey established a SLLN for the capacity of the range of SRW in dimensions $d\geq5$ \cite{jain68range}. This was then extended to a central limit theorem for $d\geq6$ by Asselah, Schapira and Sousi in \cite{asselah18capacity}. The same authors also proved a SLLN and central limit theorem for $d=4$ in \cite{asselahcapacity19}. Furthermore, in \cite{chang17two}, Chang showed that the scaling limit in $d=3$ is random.

We also give explicit expressions for the limiting values. In dimensions $d\geq5$, we show that the SLLN limit equals the non-intersection probability of a SRW and a \emph{two-sided loop-erased random walk}. Two-sided LERW for $d\geq5$ was first constructed by Lawler in \cite{lawler80selfavoiding}. Much later, Lawler, Sun and Wu constructed two-sided LERW for $d=4$ in \cite{lawler19fourdimensional}. We show that (two-sided) LERW is ergodic in four dimensions. We then use an alternative formula for the capacity to express the limit in terms of non-intersection probabilities of a SRW and two-sided LERW. Both the ergodicity of LERW for $d=4$ as well as the alternative capacity formula are of independent interest.

In three dimensions, we do not consider the two-sided LERW. Rather, we express the scaling limit of $\Ca(\eta[0,n])$ as a function of the scaling limit of LERW, which was first constructed by Kozma in \cite{kozma2007scaling}. Furthermore, we show that the correct scaling for the capacity is $n^{1/\beta}$, with $\beta$ being the \emph{growth exponent} of three-dimensional LERW (see Section \ref{sec:d3res} for a definition of the growth exponent). We then show that the three-dimensional LERW parametrized by its capacity has a scaling limit, which is Kozma's process parametrized by its capacity. This complements the work in two dimensions, where the original proofs that the scaling limit of LERW converges to the continuous process $\SLE_2$ were done under the capacity parametrization, see \cite{lawler2011conformal}. It was only shown much more recently by Lawler and Viklund in \cite{lawler2021convergence} that two-dimensional LERW also converges in the natural parametrization.

\subsection{Dimensions \texorpdfstring{$d\geq5$}{d>=5}}\label{sec:d5res}
Before stating the SLLN in high dimensions, we first need to introduce the \emph{two-sided} loop-erased random walk, which we denote by $\heta$. The random variable $\heta\colon\Z\to\Z^d$ takes values in the space of two-sided infinite simple paths. Intuitively speaking, $\heta$ can be viewed as an infinite loop-erased random walk `seen from the middle', i.e., $\heta$ is the weak limit of $T^k\eta$ as $k\rightarrow\infty$. The two-sided LERW was first introduced by Lawler in \cite[Section 5]{lawler80selfavoiding} and is defined as follows.
\begin{definition}\label{def:twosided-lerw}
Let $d\geq5$. Let $S_1,S_2$ be independent simple random walks on $\Z^d$ started from~0, conditioned such that $\LE(S_1)[1,\infty)\cap S_2[1,\infty)=\varnothing$. The two-sided loop-erased random walk $\heta$ is defined as the union of $\LE(S_1)$ and $\LE(S_2)$, i.e., for $n\geq0$, $\heta(n)=\LE(S_1)(n)$ and $\heta(-n)=\LE(S_2)(n)$.
\end{definition}
We are now ready to state the strong law of large numbers for the capacity of loop-erased random walk in dimensions $d\geq5$.
\begin{theorem}\label{th:highdim}
Let $d\geq5$. Then almost surely,
\begin{equation}
    \limn\frac{\Ca(\eta[0,n])}{n}=\P(W[1,\infty)\cap\heta(-\infty,\infty)=\varnothing\mid W(0)=\heta(0)=0).
\end{equation}
\end{theorem}

\subsection{Dimension \texorpdfstring{$d=4$}{d=4}}\label{sec:d4res}
In $\Z^4$, we have $\P(W[1,\infty)\cap\LE(S[0,\infty))=\varnothing)=0$. So we cannot define two-sided LERW as above. Note that in high dimensions, $\heta[0,\infty)$ has the law of a one-sided LERW weighted by the probability that a SRW started from the same point does not intersect it. Although this probability is 0 for an infinite SRW, we can still consider the $n$-step escape probabilities and show that they converge when properly rescaled. Let
\begin{equation}\label{eq:def-Xn-lerw}
	X_n:=(\log n)^{\frac{1}{3}}\P^W_0(W[1,n]\cap\eta[0,\infty)=\varnothing\mid\eta).
\end{equation}
\begin{theorem}[\cite{lawler19fourdimensional}]\label{th:rescaledescapeprob}
There exists a nontrivial random variable $X_\infty$ such that
\begin{equation}\label{eq:def-Xinfty-lerw}
	X_\infty=\limn X_n
\end{equation}
almost surely and in $L^p$ for all $p>0$.
\end{theorem}

We can now define two-sided LERW in the same way as in high dimensions, but using the random variable $X_\infty$ instead of $\P^W(W[1,\infty)\cap\eta[0,\infty)=\varnothing\mid\eta)$.
\begin{definition}
We define the two-sided loop-erased random walk $\heta$ as the random variable on two-sided infinite paths in $\Omega$ such that $\heta[0,\infty)$ is absolutely continuous with respect to $\eta$ with Radon-Nikodym derivative given by $\frac{X_\infty}{\E[X_\infty]}$. It was shown in \cite{lawler19fourdimensional} that $(T\heta)[0,\infty)\stackrel{d}{=}\heta[0,\infty)$. Thus, we can extend $\heta$ to the space of two-sided infinite paths by defining $\heta[-a,b]\stackrel{d}{=}\heta[0,a+b]-\heta(a)$.
\end{definition}

In order to formulate the SLLN for $d=4$, we need the following extension of Theorem \ref{th:rescaledescapeprob}.
\begin{theorem}\label{th:Xhat}
Let $d=4$. Then there exist random variables $\Xhi$ and $\Xhip$ depending on $\heta$ such that
\begin{equation}\label{eq:def-Xhat}
    \begin{split}
        \hat{X}^+_\infty&=\limn \hat{X}^+_n:=\limn(\log n)^{1/3}\P^W_0(W[1,\infty)\cap\hat{\eta}[1,n]=\varnothing\mid \heta),\\
        \hat{X}_\infty&=\limn \hat{X}_n:=\limn(\log n)^{1/3}\P^W_0(W[1,\infty)\cap\hat{\eta}[0,n]=\varnothing\mid \heta)
    \end{split}
\end{equation}
almost surely and in $L^p$ for every $p<3$.
\end{theorem}
We prove a SLLN for $d=4$ in terms of $\hat{X}_\infty$ and $\hat{X}^+_\infty$.
\begin{theorem}\label{th:d4}
Let $d=4$. Then almost surely,
\begin{equation}
    \limn\frac{(\log n)^{2/3}\Ca(\eta[0,n])}{n}=\E[\Xh_\infty\Xh^+_\infty].
\end{equation}
\end{theorem}

\subsection{Dimension \texorpdfstring{$d=3$}{d=3}}\label{sec:d3res}
Lastly, we show a non-deterministic law of large numbers for $d=3$. We relate the rescaled capacity of LERW in $\Z^d$ to the capacity of its scaling limit in $\R^d$. Existence of the scaling limit of $\eta$ as the mesh size tends to 0 was first shown by Kozma in \cite{kozma2007scaling} along a dyadic subsequence. This was later strengthened to the following statement in \cite{hernandez2024sharp}, which is what we will need for the law of large numbers for the capacity. Let $\mathcal{C}$ be the collection of parametrized curves $\lambda\colon[0,\infty)\to\R^3$ with $\lambda(0)=0$ and $\lim_{t\rightarrow\infty}\lambda(t)=\infty$. We also consider the LERW $\eta$ to be an element of $\mathcal{C}$, with linear interpolation at non-integer times. Consider the metric $\chi$ on $\mathcal{C}$ given by 
\begin{equation}\label{eq:def-chi-lerw}
    \begin{split}
        \chi(\lambda_1,\lambda_2):=\sum_{k=1}^\infty2^{-k}\sup_{0\leq t\leq k}\min\{|\lambda_1(t)-\lambda_2(t)|,1\}
    \end{split}
\end{equation}
for $\lambda_1,\lambda_2\in\mathcal{C}$. Let 
\begin{equation}\label{eq:beta}
    \beta:=\limn\frac{\log n}{\log\E\|\eta(n)\|}\in\left(1,\frac{5}{3}\right]
\end{equation}
be the \emph{growth exponent} for LERW in $\Z^3$. A priori, it is not clear that $\beta$ is well-defined. In \cite{shiraishi2018growth}, Shiraishi proved existence of the limit. Furthermore, Hern\'andez-Torres, Li and Shiraishi proved convergence of LERW to Kozma's scaling limit for the full sequence $(n^{1/\beta}\eta(n\cdot))_{n\in\N}$ in \cite{hernandez2024sharp}.
\begin{theorem}[\cite{shiraishi2018growth}, \cite{hernandez2024sharp}]\label{th:scalinglim3d}
The limit in \eqref{eq:beta} exists. Furthermore, there exists a random infinite parametrized curve $\gamma\in\mathcal{C}$ such that $(n^{-1/\beta}\eta(n\cdot))_{n\in\N}$ converges weakly to $\gamma$ in~$(\mathcal{C},\chi)$ as $n\rightarrow\infty$.
\end{theorem}

Our last main theorem expresses the scaling limit of $\Ca(\eta[0,n])$ in $\Z^3$ in terms of the capacity of $\gamma[0,1]$. First, we need to define the capacity of a subset of $\R^3$.
\begin{definition}\label{def:capacity-R3}
Let $A\subset\R^3$. Let $M$ be a standard Brownian motion and denote by $\P^M_y$ the law of $M$ with $M(0)=y$. The \emph{capacity} of $A$ is defined\footnote{As in the discrete case, there are several different equivalent definitions of capacity, but for our purposes we only need the one stated here. We refer to \cite{morters2010brownian} for more background.} as
\begin{equation}
    \Ca_{\R^3}(A):=\lim_{\|y\|\rightarrow\infty}\frac{\P_y^M(M[0,\infty)\cap A\neq\varnothing)}{G_{\R^3}(0,y)},
\end{equation}
with $G_{\R^3}$ the Green's function of Brownian motion on $\R^3$.
\end{definition}
From now on, when discussing three-dimensional capacity, we will use the subscripts $\Ca_{\Z^3}$ and $\Ca_{\R^3}$ to avoid confusion. In higher dimensions, we only consider discrete capacity and drop the subscript. We are now ready to state the law of large numbers for the capacity of three-dimensional LERW.
\begin{theorem}\label{th:d3}
Let $d=3$. Then
\begin{equation}
    \limn\frac{\Ca_{\Z^3}(\eta[0,n])}{3n^{1/\beta}}=\Cr(\gamma[0,1])
\end{equation}
in distribution.
\end{theorem}

As a consequence, we obtain convergence of the capacity-parametrized LERW. Let
\begin{equation}\label{eq:def-capacity-clocks}
    \begin{split}
        \tau_n(s):=&\,\inf\left\{t:n^{-1}\Ca_{\Z^3}(\eta[0,\lfloor t\rfloor])> s\right\},\\
        \theta(s):=&\,\inf\left\{t:\Ca_{\R^3}(\gamma[0,t])> s\right\}
    \end{split}
\end{equation}
be the \emph{capacity clocks} of $\eta$ and $\gamma$ respectively. Define the \emph{capacity-parametrized} LERW $\widetilde{\eta}_n$ and scaling limit $\widetilde{\gamma}$ as
\begin{equation}\label{eq:def-capacity-param-proc}
    \begin{split}
        \widetilde{\eta}_n(s):=n^{-1}\eta(\tau_n(s)),\\
        \widetilde{\gamma}(s):=\gamma(3^{-\beta}\theta(s)).
    \end{split}
\end{equation}
We prove the following result.
\begin{theorem}\label{th:scaling-limit-cap-par}
Let $d=3$. Then $(\widetilde{\eta}_n)_{n\in\N}$ converges weakly to $\widetilde{\gamma}$ with respect to $(\mathcal{C},\chi)$ as $n\rightarrow\infty$.
\end{theorem}

\subsection{Proof sketch}\label{sec:sketch-lerw}
\subsubsection{Dimensions \texorpdfstring{$d\geq5$}{d>=5}}
The proof of the law of large numbers in high dimensions is the most straightforward one. Recall that the capacity can be written as
\begin{equation}
	\frac{1}{n}\Ca(\eta[0,n])=\frac{1}{n}\sum_{k=0}^n\P^W_{\eta(k)}(W[1,\infty)\cap\eta[0,n]=\varnothing\mid\eta).
\end{equation}
If the probabilities in the sum above were i.i.d., a strong law of large numbers would follow immediately. This is of course not the case, but we can still apply an ergodic theorem to deduce the law of large numbers as follows. As $k\rightarrow\infty$, the part of $\eta$ near $\eta(k)$ is distributed as $\heta$. So $\frac{1}{n}\Ca(\eta[0,n])$ and $\frac{1}{n}\Ca(\heta[0,n])$ should have the same limit, which we show in Proposition \ref{prop:reductiontwosided}. Furthermore, $\heta$ is stationary and ergodic. So by Birkhoff's ergodic theorem \cite[Theorem 24.1]{billingsley1995probability},
\begin{equation}
	\begin{split}
		\frac{1}{n}\Ca(\heta[0,n])=&\frac{1}{n}\sum_{k=0}^n\P^W_0(W[1,\infty)\cap(T^k\heta)[-k,n-k]\mid\heta)\\
		\approx&\frac{1}{n}\sum_{k=0}^n\P^W_0(W[1,\infty)\cap(T^k\heta)(-\infty,\infty)\mid\heta)\\
		\rightarrow&\P(W[1,\infty)\cap\heta(-\infty,\infty)=\varnothing\mid W(0)=\heta(0)=0),\qquad n\rightarrow\infty,
	\end{split}
\end{equation}
which was what we wanted to show. The approximate equality comes from the fact that in dimensions $d\geq5$, a LERW and a SRW only intersect each other near their starting points.

\subsubsection{Dimension \texorpdfstring{$d=4$}{d=4}}
The overall picture in four dimensions is the same as in higher dimensions, but there are several subtleties. Dimension 4 is critical for the intersection of LERW and SRW. This means that an infinite SRW and an infinite LERW started from the same point intersect almost surely, and the probability that the first $n$ steps intersect decays polylogarithmically in $n$. So while $\frac{\Ca(\eta[0,n])}{n}\rightarrow0$ almost surely as $n\rightarrow\infty$, we obtain a non-trivial limit if we multiply by a polylogarithmic correction. It turns out that the correct scaling is $\frac{(\log n)^{2/3}}{n}$, which we will now explain.

Like in high dimensions, the strong law of large numbers for $\heta$ is the same as for $\eta$, so we only consider two-sided LERW. We first show in Lemma \ref{lem:Xconvergencetwosided} that the limits
\begin{equation}\label{eq:sketchXinf}
	\begin{split}
		\hat{X}_{\infty}=&\limn\hat{X}_n:=\limn(\log n)^{1/3}\P_0^W(W[1,\infty)\cap\heta[0,n]=\varnothing\mid\heta)\\
	\hat{X}^+_{\infty}=&\limn\hat{X}^+_n:=\limn(\log n)^{1/3}\P_0^W(W[1,\infty)\cap\heta[1,n]=\varnothing\mid\heta)\\
	\end{split}
\end{equation}
exist, building on a very similar result from \cite{lawler19fourdimensional}. In other words, the intersection probability of $W[1,\infty)$ and $\heta[0,n]$ decays like $(\log n)^{-1/3}$. So the intersection probability of $W[1,\infty)$ and $\heta[-n,n]$ should decay like $(\log n)^{-2/3}$. So
\begin{equation}
	\Ca(\eta[0,n])=\sum_{k=0}^n\P_{\heta(k)}^W(W[1,\infty)\cap\heta=\varnothing\mid\heta)\asymp(\log n)^{-2/3}n.
\end{equation}

Note that the escape probabilities in the expression above are \emph{two-sided} escape probabilities: $W$ started from $\eta(k)$ has to escape both $\eta[0,k]$ and $\eta[k+1,n]$. However, we only know the scaling limits of the one-sided escape probabilities from \eqref{eq:sketchXinf}. So in order to obtain a precise expression for the limit, we use a different formula for the capacity. In Lemma \ref{lem:capacitydecomp}, we show that\footnote{This formula for the capacity was known to some people in the field (e.g. Asselah \cite{asselahprivate}), but the author has not seen this precise expression anywhere in the existing literature. A related expression appears in \cite{losev2023long}.}
\begin{equation}
	\begin{split}
		\Ca(\heta[0,n])
        =&\sum_{k=0}^n\P^W_{\heta(k)}(W[1,\infty)\cap\heta[k,n]=\varnothing\mid\heta)\P^W_{\heta(k)}(W[1,\infty)\cap\heta[k+1,n]=\varnothing\mid\heta)
	\end{split}
\end{equation}
In other words, instead of writing the capacity as the sum of two-sided escape probabilities, we write the capacity as the sum of the product of two one-sided escape probabilities. This can be rewritten further in terms of $\hat{X}$ and $\hat{X}^+$:
\begin{equation}
	\begin{split}
		\Ca(\heta[0,n])=&\sum_{k=0}^n(\log k)^{-2/3}\hat{X}_{n-k}(T^k\heta)\hat{X}_{n-k}^+(T^k\heta)\\
		\approx&\sum_{k=0}^n(\log n)^{-2/3}\hat{X}_{\infty}(T^k\heta)\hat{X}_{\infty}^+(T^k\heta).
	\end{split}
\end{equation}
The approximate equality is due to the fact that $W[1,\infty)$ is most likely to hit $\heta[0,n]$ close to the origin. The desired strong law of large numbers now follows from ergodicity and stationarity of $\heta$ with respect to $T$, which we prove in Proposition \ref{prop:ergodicity4}.

\subsubsection{Dimension \texorpdfstring{$d=3$}{d=3}}
The picture in three dimensions is entirely different. The capacity of LERW no longer scales (almost) linearly and the limit law is no longer deterministic. In three dimensions, the capacity of a set is of the same order as its diameter. Since the LERW is roughly $\beta$-dimensional, the correct scaling for $\Ca(\eta[0,n])$ is $n^{1/\beta}$, where $\beta$ is the growth exponent of LERW as defined in Section \ref{sec:d3res}. This also means that the global behaviour of the LERW influences the capacity, which implies that the scaling limit must be non-deterministic. The proof goes as follows. We use the same strategy as in \cite{chang17two}.

Recall that $\gamma$ denotes the scaling limit of LERW in three dimensions. For a set $A$ in~$\Z^3$ or $\R^3$, let $B(A,\varepsilon):=\{x\colon\inf_{a\in A}\|x-a\|\leq\varepsilon\}$ denote the $\varepsilon$-sausage of $A$. We first show that a SRW started from distance $\delta n^{1/\beta}$ from $\eta[0,n]$ is very likely to hit $\eta[0,n]$ for $\delta$ small enough. So the probability that a SRW hits $\eta[0,n]$ is close to the probability that a SRW hits $B(\eta[0,n],\delta n^{1/\beta})$. This implies that 
\begin{equation}\label{eq:step1sketch}
	\Ca_{\Z^3}(\eta[0,n])\approx \Ca_{\Z^3}(B(\eta[0,n],\delta n^{1/\beta})).
\end{equation}
Furthermore, $n^{-1/\beta}\eta[0,n]$ converges in distribution to $\gamma[0,1]$. So we can couple $\eta$ and $\gamma$ in such a way that $B(n^{-1/\beta}\eta[0,n],\delta)\approx B(\gamma[0,1],\delta)$, which implies
\begin{equation}\label{eq:step2sketch}
	\Ca_{\Z^3}(B(\eta[0,n],\delta n^{1/\beta}))\approx\Ca_{\Z^3}(B(n^{1/\beta}\gamma[0,1],\delta n^{1/\beta})).
\end{equation}
Using a strong coupling of SRW and Brownian motion and the fact that $G_{\R^3}(0,y)\sim3G_{\Z^3}(0,y)$, we then show
\begin{equation}\label{eq:step3sketch}
	\Ca_{\Z^3}(B(n^{1/\beta}\gamma[0,1],\delta n^{1/\beta}))\approx3\Ca_{\R^3}(B(n^{1/\beta}\gamma[0,1],\delta n^{1/\beta}))=3n^{1/\beta}\Ca_{\R^3}(B(\gamma[0,1],\delta)).
\end{equation}
Finally, for $\delta$ small, a Brownian motion started from the boundary of $B(\gamma[0,1],\delta)$ will hit $\gamma[0,1]$ with high probability. Thus,
\begin{equation}\label{eq:step4sketch}
	3n^{1/\beta}\Ca_{\R^3}(B(\gamma[0,1],\delta))\approx3n^{1/\beta}\Ca_{\R^3}(\gamma[0,1]).
\end{equation}
The desired law of large numbers then follows by combining \eqref{eq:step1sketch}--\eqref{eq:step4sketch}.

In order to conclude Theorem~\ref{th:scaling-limit-cap-par}, we first strengthen Theorem~\ref{th:d3} to convergence of the entire process $(n^{-1/\beta}\Cz(\eta[0,\lfloor nu\rfloor]))_{u\geq0}$ with respect to the measure~$\chi$. Pointwise convergence follows from Theorem~\ref{th:d3}. Strengthening this to uniform convergence on compacts follows from Arzel\`a-Ascoli, where we use that the difference in capacity of two segments of the LERW is at most a constant times the diameter. It is then a straightforward exercise to conclude the proof of Theorem~\ref{th:scaling-limit-cap-par}. We use that $\Cr(\gamma[0,\cdot])$ is a strictly increasing function, which means it is a continuity point of the inverse function. The factor $3^{-\beta}$ in the definition of $\widetilde\gamma$ is needed to account for the fact that discrete and continuous capacity differ by a factor 3.

\subsection{Outline}\label{sec:outline-lerw}
In Section \ref{sec:highdim}, we give the proof of Theorem \ref{th:highdim}. We prove Theorem \ref{th:d4} in Section \ref{sec:d4}. Finally, we give the proof of Theorems \ref{th:d3} and \ref{th:scaling-limit-cap-par} in Section \ref{sec:d3proof}.

\section{Capacity of LERW in high dimensions}\label{sec:highdim}
In this section, we consider loop-erased random walk in dimensions $d\geq5$ and prove Theorem~\ref{th:highdim}. We first state some facts about two-sided LERW. We then prove Theorem \ref{th:highdim} by showing a lower and upper bound.
\begin{proposition}[\cite{lawler80selfavoiding}]
Two-sided LERW $\heta$ satisfies the following properties:
\begin{enumerate}
	\item the law $\heta$ is absolutely continuous with respect to the law of $\eta$;
	\item $T^k\eta$ converges weakly to $\heta$ as $k\rightarrow\infty$;
	\item $\heta$ is stationary with respect to $T$, i.e., $T\heta\stackrel{d}{=}\heta$;
	\item $\eta$ and $\heta$ are ergodic with respect to $T$.
\end{enumerate}
\end{proposition}
As a corollary, we show that it suffices to prove a strong law of large numbers for $\Ca(\heta[0,n])$.
\begin{proposition}\label{prop:reductiontwosided}
Assume that $\frac{1}{n}\Ca(\heta[0,n])$ converges almost surely. Then $\frac{1}{n}\Ca(\eta[0,n])$ converges almost surely to the same limit.
\end{proposition}
\begin{proof}
Note that the law of $\heta$ is absolutely continuous with respect to the law of $\eta$. So if $\frac{1}{n}\Ca(\heta[0,n])$ converges almost surely, $\frac{1}{n}\Ca(\heta[0,n])$ converges with positive probability. Since $\eta$ is ergodic with respect to $T$ and $\frac{1}{n}\Ca(\heta[0,n])$ converges if and only if $\frac{1}{n}\Ca((T\heta)[0,n])$ converges, the probability must equal 1.
\end{proof}

\begin{proof}[Proof of Theorem \ref{th:highdim}]
By Proposition \ref{prop:reductiontwosided}, it suffices to prove a strong law of large numbers for the two-sided LERW. We first prove a lower bound. By the definition of capacity and recalling that $\P^W$ is the law of a SRW $W$, we have 
\begin{equation}
	\begin{split}
		\Ca(\heta[0,n])=&\sum_{k=0}^n\P_{\heta(k)}^W(W[1,\infty)\cap\heta[0,n]=\varnothing\mid\heta)\\
		\geq&\sum_{k=0}^n\P_{\heta(k)}^W(W[1,\infty)\cap\heta(-\infty,\infty)=\varnothing\mid\heta)\\
		=&\sum_{k=0}^n\P_{0}^W(W[1,\infty)\cap(T^k\heta)(-\infty,\infty)=\varnothing\mid\heta).
	\end{split}
\end{equation}
Since $\heta$ is ergodic and stationary with respect to $T$, it follows by Birkhoff's ergodic theorem \cite[Theorem 24.1]{billingsley1995probability} that almost surely,
\begin{equation}
	\begin{split}
	\liminf_{n\rightarrow\infty}\frac{1}{n}\Ca(\heta[0,n])=&\E^{\heta}[\P_{0}^W(W[1,\infty)\cap(T^k\heta)(-\infty,\infty)=\varnothing\mid\heta)]\\
		=&\P(W[1,\infty)\cap\heta(-\infty,\infty)=\varnothing\mid W(0)=\heta(0)=0),
	\end{split}
\end{equation}
which settles the lower bound.

We now turn to the upper bound. Let $m\in\N$. It can be seen immediately from Definition \ref{def:capacity} that capacity is subadditive. Thus,
\begin{equation}
	\begin{split}
		\frac{1}{n}\Ca(\heta[0,n])\leq&\frac{1}{n}\sum_{k=0}^{\lceil\frac{n}{m}\rceil-1}\Ca(\heta[mk,m(k+1)])\\
		=&\frac{1}{m}\frac{1}{\frac{n}{m}}\sum_{k=0}^{\lceil\frac{n}{m}\rceil-1}\Ca\left(((T^m)^k\heta)[0,m]\right)
	\end{split}
\end{equation}
From the proofs in \cite{lawler80selfavoiding}, it follows that $\heta$ is also ergodic with respect to any power of $T$, so also with respect to $T^m$. Therefore, again by Birkhoff's ergodic theorem,
\begin{equation}
	\begin{split}
		\limsup_{n\rightarrow\infty}\frac{1}{n}\Ca(\heta[0,n])\leq\frac{1}{m}\E[\Ca(\heta[0,m])].
	\end{split}
\end{equation}
Therefore, it suffices to show 
\begin{equation}
	\limm\E[\Ca(\heta[0,m])]=\P(W[1,\infty)\cap\heta(-\infty,\infty)=\varnothing\mid W(0)=\heta(0)=0).
\end{equation}
Note that as $k,m-k\rightarrow\infty$, we have $\P(W[1,\infty)\cap(\heta(-\infty,-k)\cup\heta(m-k,\infty))\neq\varnothing)\rightarrow0$. So indeed, by stationarity of $\heta$,
\begin{equation}
	\begin{split}
		\E[\Ca(\heta[0,m])]=&\frac{1}{m}\sum_{k=0}^m\P(W[1,\infty)\cap\heta[0,m]=\varnothing\mid W(0)=\heta(k))\\
		=&\frac{1}{m}\sum_{k=0}^m\P(W[1,\infty)\cap\heta[-k,m-k]=\varnothing\mid W(0)=\heta(0)=0)\\
		=&\frac{1}{m}\sum_{k=0}^m\P(W[1,\infty)\cap\heta(-\infty,\infty)=\varnothing\mid W(0)=\heta(0)=0)+o(1)\\
		=&\P(W[1,\infty)\cap\heta(-\infty,\infty)=\varnothing\mid W(0)=\heta(0)=0)+o(1),\qquad m\rightarrow\infty,
	\end{split}
\end{equation}
which concludes the proof of Theorem \ref{th:highdim}.
\end{proof}

\section{Capacity of LERW in four dimensions}\label{sec:d4}
In this section, we consider LERW in four dimensions. We first show that two-sided LERW is ergodic in four dimensions in Section \ref{sec:twosided4}. We then prove Theorem \ref{th:d4} in Section~\ref{sec:proof4}.
\subsection{Ergodicity of two-sided LERW}\label{sec:twosided4}
\begin{proposition}\label{prop:ergodicity4}
The random variable $\heta$ is ergodic with respect to $T$. 
\end{proposition}
Throughout the paper, we mostly consider LERW $\eta$ as its own process, without reference to the underlying SRW $S$. However, for this proof, we will need several facts about $S$ and the loop-erasing procedure. Recall the definition of the loop-erasure times~$\l_i$ from Definition \ref{def:lerw}. Also define $\rho_j=\max\{i\geq 0\colon\l_i\leq j\}$, which is the number of points up until time $j$ that are kept in the loop-erasure. We say $n$ is a \emph{cut time} of the walk $S$ if $S[0,n]\cap S[n+1,\infty)=\varnothing$. The behaviour of $\l_i$, $\rho_j$ and the cut times of $S$ in four dimensions are well understood. In particular, $\l_i$ and $\rho_j$ are well concentrated around $i(\log i)^{1/3}$ and $j(\log j)^{-1/3}$ and $S$ has a large amount of cut times. We summarize the results we need in the following lemma, whose bounds are implicitly stated in the proofs of \cite[Lemma~7.7.4 and Theorem 7.7.5]{lawler91intersections}.
\begin{lemma}\label{lemma:cuttimes}
We have
\begin{equation}
	\begin{split}
		\limn\E\left[|n-\rho_{n(\log n)^{-1/3}}|\right]=\limn\E\left[|(\log n)^{1/3}n-\l_n|\right]=0.
	\end{split}
\end{equation}
Furthermore, there exists $C>0$ such that for all $n$ and all $m$ such that $|n-m|\geq (\log m)^{-6}m$,
\begin{equation}
	\P(\text{there are no cut times between $n$ and $m$})\leq C\frac{\log\log m}{\log m}.
\end{equation}
\end{lemma}

\begin{proof}[Proof of Proposition \ref{prop:ergodicity4}]
The proof is an adaptation of \cite[Proposition 5.9]{lawler80selfavoiding}. Let $\zeta=(\zeta_0,\ldots,\zeta_m)$, $\xi=(\xi_0,\ldots,\xi_{m})$ be finite self-avoiding paths. Since $\heta$ is absolutely continuous with respect to $\eta$, it suffices to prove that $\heta$ is ergodic with respect to $T$. Since $\cF$ is generated by cylinder sets, it suffices to show that for all finite self-avoiding paths $\zeta=(\zeta_0,\ldots,\zeta_m)$ and $\xi=(\xi_0,\ldots,\xi_{m})$ \cite[Lemma 6.7.4]{gray09probability},
\begin{equation}
\lim_{N\rightarrow\infty}\frac{1}{N}\sum_{k=1}^N\P(\eta[0,m]=\zeta,\eta[k,k+m]=\xi)=\lim_{N\rightarrow\infty}\frac{1}{N}\sum_{k=1}^N\P(\eta[k,k+m]=\xi)\P(\eta[0,m]=\zeta).
\end{equation}
The domain Markov property for LERW states that the tail of a LERW conditioned to start with a certain path is distributed as the loop-erasure of a SRW conditioned to avoid the path \cite[Section 7.3]{lawler91intersections}. Thus,
\begin{equation}
    \begin{split}
        \P(\eta[k,k+m]=\xi\mid\eta[0,m]=\zeta)=&\P(\eta[k-m,k]=\xi\mid S[1,\infty)\cap\zeta=\varnothing,\,S_0=\zeta_m),
    \end{split}
\end{equation}
where $S$ is the SRW such that $\eta=\LE(S)$. Let $F$ be the event that $S[1,\infty)\cap\zeta=\varnothing$. Then it suffices to prove that for every $\xi$ and $\zeta$ such that $F$ has strictly positive probability,
\begin{equation}
    \lim_{N\rightarrow\infty}\frac{1}{N}\sum_{k=1}^N\P(\eta[k,k+m]=\xi\mid F)=\lim_{N\rightarrow\infty}\frac{1}{N}\sum_{k=1}^N\P(\eta[k,k+m]=\xi).
\end{equation}
Let
\begin{equation}
    H^n(\xi)=\frac{1}{n}\sum_{k=0}^n\1_{\{\eta_{k}^{k+m}=\xi\}}.
\end{equation}
We divide the interval $[0,n(\log n)^{1/3}]$ into $(\log n)^{1/2}$ intervals $I_j$ of size $n(\log n)^{-1/6}$, separated by buffers $L_j$ and $R_j$ of size $n(\log n)^{-6}$ as follows. Let $a_j=\lfloor jn(\log n)^{-1/6}\rfloor$, $a_j^-=a_j-\lfloor n(\log n)^{-6}\rfloor$, and $a_j^+=a_j+\lfloor n(\log n)^{-6}\rfloor$. Define $L_j=[a_j,a_j^+]$, $I_j=[a_j^+,a_{j+1}^-]$ and $R_j=[a_{j+1}^-,a_{j+1}]$.\footnote{The precise size and amount of intervals do not really matter. The important thing is that with high probability, all buffers contains a cut time.} Let 
\begin{equation}
    H^n_j(\xi)=\sum_{k=\rho_{a_j}}^{\rho_{a_{j+1}}-1}\1_{\eta_k^{m+k}=\xi},\qquad \Bar{H}^n_j(\xi)=\sum_{k=0}^{\lfloor a_1(\log a_1)^{-1/3}\rfloor}\1_{\{\LE(S[a_j,a_{j+1}])_k^{k+m}=\xi\}}.
\end{equation}
We show that we can approximate $H^n(\xi)$ by the sum over $\bar{H}_j^n(\xi)$. Firstly,
\begin{equation}
    \begin{split}
        &\left|\E\left[H^n(\xi)-\frac{1}{n}\sum_{j=0}^{\lfloor (\log n)^{1/2}\rfloor}\Bar{H}_j^n(\xi)\mid F\right]\right|\\
        \leq&\left|\E\left[H^n(\xi)-\frac{1}{n}\sum_{k=0}^{\rho_{n(\log n)^{1/3}}}\1_{\{\eta_{k}^{k+m}=\xi\}}\mid F\right]\right|\\
        &+\left|\E\left[\frac{1}{n}\sum_{k=0}^{\rho_{n(\log n)^{1/3}}}\1_{\{\eta_{k}^{k+m}=\xi\}}-\frac{1}{n}\sum_{j=0}^{\lfloor (\log n)^{1/2}\rfloor}\Bar{H}_j^n(\xi)\mid F\right]\right|\\
        \leq&\frac{1}{n}\E[|n-\rho_{n(\log n)^{1/3}}|]\P(F)^{-1}\\
        &+\E\left[\frac{1}{n}
        \sum_{j=0}^{\lfloor(\log n)^{1/2}\rfloor}|H_j^n(\xi)-\Bar{H}_j^n(\xi)|\1_{\{\text{Each $L_j$ and $R_j$ contains a cut time}\}}\right]\P(F)^{-1}\\
        &+\P(\text{There is a $j$ such that $L_j$ or $R_j$ does not contain a cut time})\P(F)^{-1}
    \end{split}
\end{equation}
Consider the event that $L_j$ and $R_j$ contains a cut time. Then $\eta[\rho_{a_j},\rho_{a_{j+1}}]$ and $\LE(S[a_j,a_{j+1}])$ only differ on the sections contributed by $L_j$ and $R_j$. The number of indices of $L_j$ contributes to $\eta$ is $\rho_{a_j^+}-\rho_{a_j}\leq a_j^+-a_j=n(\log n)^{-6}$. The contribution of $L_j$ to $\LE(S[a_j,a_{j+1}])$ is also at most $a_j^+-a_j$. The same estimates hold for $R_j$. Furthermore, the number of indices counted by $H_j^n(\xi)$ and $\Bar{H}_j^n(\xi)$ differs by at most $|\rho_{a_{j+1}}-\rho_{a_j}-\lfloor a_1(\log a_1)^{-1/3}\rfloor|$. By Lemma \ref{lemma:cuttimes}, the probability that $L_j$ or $R_j$ does not contain a cut time is bounded by~$\lesssim\frac{\log\log n}{\log n}$ and $\E|n-\rho_{n(\log n)^{-1/3}}|=o(1)$. Thus, the above is bounded by
\begin{equation}
    \begin{split}
        \lesssim&o(1)+\frac{1}{n}\sum_{j=0}^{\lfloor(\log n)^{1/2}\rfloor}\left(2n(\log n)^{-6}+\E[|\rho_{a_{j+1}}-\rho_{a_j}-\lfloor a_1(\log a_1)^{-1/3}\rfloor|]\right)\\
        &+\frac{\log\log n(\log n)^{1/2}}{\log n}\\
        \lesssim&o(1)+(\log n)^{-5}+\frac{1}{n}
        \sum_{j=0}^{\lfloor(\log n)^{1/2}\rfloor}a_{1}(\log a_1)^{-1/3}o(a_1)+\frac{\log\log n}{(\log n)^{1/2}}\rightarrow0,\qquad n\rightarrow\infty,
    \end{split}
\end{equation}
where $f\lesssim g$ denotes that $f$ is bounded from above by $g$ times some universal constant that is independent of the domain of $f$ and $g$. So it suffices to show that
\begin{equation}
	\limn\E\left[\frac{1}{n}\sum_{j=0}^{\lfloor (\log n)^{1/2}\rfloor}\Bar{H}_j^n(\xi)\mid F\right]=\limn\E\left[\frac{1}{n}\sum_{j=0}^{\lfloor (\log n)^{1/2}\rfloor}\Bar{H}_j^n(\xi)\right].
\end{equation}
First note that we can ignore the term for $j=0$. Furthermore, as $m\rightarrow\infty$, the law of $S[m,\infty)-S_m$ conditioned to avoid a finite set $A$ converges in total variation norm to the law of $S[0,\infty)$. So
\begin{equation}
    \begin{split}
        &\left|\E\left[\frac{1}{n}\sum_{j=1}^{\lfloor (\log n)^{1/2}\rfloor}\Bar{H}_j^n(\xi)\mid F\right]-\E\left[\frac{1}{n}\sum_{j=1}^{\lfloor (\log n)^{1/2}\rfloor}\Bar{H}_j^n(\xi)\right]\right|\\
        \leq&\frac{1}{n}\sum_{j=1}^{\lfloor(\log n)^{1/2}\rfloor}\sum_{k=0}^{a_1(\log a_1)^{-1/3}}|\P(\LE(S[\lfloor a_j\rfloor,\lfloor a_{j+1}\rfloor])_k^{k+m}=\xi\mid F)\\
        &-\P(\LE(S[\lfloor a_j\rfloor,\lfloor a_{j+1}\rfloor])_k^{k+m}=\xi)|\\\rightarrow&0, \qquad n\rightarrow\infty.
    \end{split}
\end{equation}
We conclude that $|\E[H(n,\xi)\mid F]-\E[H(n,\xi)]|\rightarrow0$ as $n\rightarrow\infty$, and so that 
\begin{equation}\label{eq:ergodicitysimple}
    \limn\left|\frac{1}{n}\sum_{k=0}^n\P(\eta[k,k+m]=\xi\mid\eta[0,m]=\zeta)-\frac{1}{n}\sum_{k=0}^n\P(\eta[k,k+m]=\xi)\right|=0
\end{equation}
for all self-avoiding paths $\zeta,\xi$. It only remains to show that, $\frac{1}{n}\sum_{k=0}^n\P(\eta[k,k+m]=\xi)$ converges. Let $f$ denote the Radon-Nikodym derivative of $\heta$ with respect to $\eta$ and let $\E^{\eta}$ denote expectation with respect to $\eta$. Then
\begin{equation}
    \begin{split}
        \P(\heta[0,m]=\xi)=&\frac{1}{n}\sum_{k=1}^n\P(\heta[k,k+m]=\xi)\\
        =&\frac{1}{n}\sum_{k=1}^n\E^{\eta}[f\1_{\eta[k,k+m]=\xi}].
    \end{split}
\end{equation}
Let $\varepsilon>0$. Since $f$ is integrable, there exists a simple function $\sum_i\alpha_i\1_{A_i}$ such that $\E^{\eta}[|f-\sum_{i}\alpha_i\1_{A_i}|]<\varepsilon$ and $A_i$ are cylinder sets. By the above, there exists $N$ such that for all $n\geq N$, we have for all $i$, $|\frac{1}{n}\sum_{k=0}^n\E[\1_{\eta[k,k+m]=\xi}\1_{A_i}]-\frac{1}{n}\sum_{k=0}^n\P(\eta[k,k+m]=\xi)\E[\1_{A_i}]|<\varepsilon$. So for all $n\geq N$, by stationarity of $\heta$, the above equals
\begin{equation}
    \begin{split}
        \sum_i\alpha_i\frac{1}{n}\sum_{k=1}^n\E^{\eta}[\1_{\eta[k,k+m]=\xi}\1_{A_i}]\pm\varepsilon
        =&\sum_i\alpha_i\frac{1}{n}\sum_{k=1}^n\P(\eta[k,k+m]=\xi)\P(\eta\in A_i)\pm2\varepsilon\\
        =&\frac{1}{n}\sum_{k=1}^n\P(\eta[k,k+m]=\xi)\pm3\varepsilon.
    \end{split}
\end{equation}
It follows that $\frac{1}{n}\sum_{k=1}^n\P(\eta[k,k+m]=\xi)\rightarrow\P(\heta[0,m]=\xi)$ as $n\rightarrow\infty$, which concludes the proof.
\end{proof}

\subsection{Proof of Theorem \ref{th:d4}}\label{sec:proof4}
We now prove Theorem \ref{th:d4}. Recall from Proposition \ref{prop:reductiontwosided} that it suffices to prove the strong law of large numbers for $\heta$. We first show the following key lemma, which holds in all dimensions. The proof only uses the strong Markov property of SRW.
\begin{lemma}\label{lem:capacitydecomp}
Let $A=\{x_1,\ldots,x_n\}\subset\Z^d$. Then
\begin{equation}
    \Ca(A)=\sum_{k=1}^n\P_{x_k}(W[1,\infty)\cap\{x_1,\ldots,x_k\}=\varnothing)\P_{x_k}(W[1,\infty)\cap\{x_1,\ldots,x_{k-1}\}=\varnothing).
\end{equation}
\end{lemma}
\begin{proof}
For $B\subset\Z^d$, let $\tau_B:=\min\{t\geq0\colon W(t)\in B\}$ be the first hitting time of $B$. The proof is by induction. The statement is clearly true for all sets of size 1. Now assume the statement holds for all sets of size $n-1$. Then
\begin{equation}
    \begin{split}
        \Ca(A)=&\sum_{k=1}^n\P_{x_k}(W[1,\infty)\cap\{x_1,\ldots,x_n\}=\varnothing)\\
        =&\sum_{k=1}^n[\P_{x_k}(W[1,\infty)\cap\{x_1,\ldots,x_{n-1}\}
        =\varnothing)\\
        &-\P_{x_k}(W[1,\infty)\cap\{x_1,\ldots,x_{n-1}\}=\varnothing,\,W[1,\infty)\cap\{x_n\}\neq\varnothing)]\\
        =&\Ca(\{x_1,\ldots,x_{n-1}\})+\P_{x_n}(W[1,\infty)\cap\{x_1,\ldots,x_{n-1}\}=\varnothing)\\
        &-\sum_{k=1}^n\P_{x_k}(\tau^+_{x_n}=\tau^+_{\{x_1,\ldots,x_{n}\}}<\infty)\P_{x_n}(W[1,\infty)\cap\{x_1,\ldots,x_{n-1}\}=\varnothing)\\
        =&\Ca(\{x_1,\ldots,x_{n-1}\})\\
        &+\P_{x_n}(W[1,\infty)\cap\{x_1,\ldots,x_{n-1}\}=\varnothing)\left(1-\sum_{k=1}^n\P_{x_n}(\tau_{x_k}^+=\tau_{\{x_1,\ldots,x_n\}}^+<\infty)\right)\\
        =&\Ca(\{x_1,\ldots,x_{n-1}\})\\
        &+\P_{x_n}(W[1,\infty)\cap\{x_1,\ldots,x_{n-1}\}=\varnothing)(1-\P_{x_n}(W[1,\infty)\cap\{x_1,\ldots,x_{n}\}\neq\varnothing))\\
        =&\Ca(\{x_1,\ldots,x_{n-1}\})\\
        &+\P_{x_n}(W[1,\infty)\cap\{x_1,\ldots,x_{n-1}\}=\varnothing)\P_{x_n}(W[1,\infty)\cap\{x_1,\ldots,x_{n}\}=\varnothing)\\
        =&\sum_{k=1}^n\P_{x_k}(W[1,\infty)\cap\{x_1,\ldots,x_k\}=\varnothing)\P_{x_k}(W[1,\infty)\cap\{x_1,\ldots,x_{k-1}\}=\varnothing),
    \end{split}
\end{equation}
which completes the proof.
\end{proof}

The above lemma is very useful, since it rewrites the capacity as the sum of one-sided escape probabilities. This allows us to use Theorem \ref{th:rescaledescapeprob}. First of all, we define
\begin{equation}
	X^+_\infty:=\limn X^+_n:=\limn(\log n)^{1/3}\P^W_0(W[1,n)\cap\eta[1,\infty)=\varnothing\mid\eta),\\
\end{equation}
whose existence was also proved in \cite{lawler19fourdimensional}. Note that $X_n$ and $X^+_n$ are defined as the non-intersection probability of an $n$-step SRW and an infinite LERW. However, for our purposes, we need to study the non-intersection probability of an infinite SRW and an $n$-step LERW. We first show that these have the same scaling limit.
\begin{lemma}\label{lem:convergenceXtilde}
Let 
\begin{equation}
    \begin{split}
        \widetilde{X}^+_n=&(\log n)^{1/3}\P_0^W(W[1,\infty)\cap\eta[1,n]=\varnothing\mid \eta),\\
        \widetilde{X}_n=&(\log n)^{1/3}\P_0^W(W[1,\infty)\cap\eta[0,n]=\varnothing\mid \eta),
    \end{split}
\end{equation}
Then $X_\infty=\limn \widetilde{X}_n$ and $X_\infty^+=\limn\widetilde{X}_n^+$ almost surely and in $L^p$ for every $p<3$.\footnote{It should be possible to prove $L^p$ convergence for all $p>0$, but we do not need it}
\end{lemma}
In the proof of this lemma, we will need certain hitting estimates on four-dimensional SRWs. Because they are only used in this proof, we will simply refer the reader to the relevant literature when needed, and not restate them in this paper. We will also need the following analysis lemma, which will be used multiple times in this paper. The lemma allows us to strengthen subsequential convergence to full convergence under certain conditions.
\begin{lemma}\label{lem:subseqconv}
Let $(a_n)_{n\in\N},(b_n)_{n\in\N}$ be strictly increasing sequences and let $f$ be a decreasing function such that
\begin{equation}
	\limn b_{a_n}f(a_n)=x.
\end{equation}
For $n\in\N$, let $k_n$ be the index such that $a_{k_n}\leq n\leq a_{k_{n}+1}$. Then if $\limn\frac{b_{a_{k_n}}}{b_{a_{k_n+1}}}=1$, we have
\begin{equation}
	\limn b_n f(n)=x.
\end{equation}
\end{lemma}
\begin{proof}
Since $b$ is increasing, $f$ is decreasing and $a_{k_n}\leq n\leq a_{k_n+1}$, we have
\begin{equation}
	\frac{b_{a_{k_n+1}}}{b_{a_{k_n}}}b_{a_{k_n}}f(a_{k_n})\geq b_nf(n)\geq\frac{b_{a_{k_n}}}{b_{a_{k_n+1}}}b_{a_{k_n+1}}f(a_{k_n+1}).
\end{equation}
The result follows by taking the limit $n\rightarrow\infty$.
\end{proof}

\begin{proof}[Proof of Lemma \ref{lem:convergenceXtilde}]
We prove the statement for $\widetilde{X}^+_n$, the proof for $\widetilde{X}_n$ being identical. We first prove almost sure convergence. We show an upper and lower bound. 
Let $S$ be the simple random walk such that $\eta=\LE(S[0,\infty)$. Then
\begin{equation}
    \begin{split}
        1-X^+_n\leq&(\log n)^{1/3}\P^W_0(W[1,n]\cap\eta[0,n]\neq\varnothing\mid\eta)\\
        &+(\log n)^{1/3}\P^W_0(W[1,n]\cap\eta[n+1,\infty)\neq\varnothing\mid\eta)\\
        \leq&1-\widetilde{X}^+_n+(\log n)^{1/3}\P_0^W(W[1,n]\cap S[n+1,\infty)\neq\varnothing\mid S).
    \end{split}
\end{equation}
Let $\varepsilon>0$. By \cite[Theorem 4.3.6]{lawler91intersections}, 
\begin{equation}
	\E^S_0[\P^W_0(W[1,n]\cap S[n+1,\infty)\neq\varnothing\mid S)]\lesssim(\log n)^{-1}.
\end{equation}
Hence, by Markov's inequality,
\begin{equation}
	\P^S_0((\log n)^{1/3}\P^W_0(W[1,n]\cap S[n+1,\infty)\neq\varnothing\mid S)>\varepsilon)\lesssim\varepsilon(\log n)^{-2/3}.
\end{equation}
Let $a_n=e^{n^2}$. Then
\begin{equation}
    \begin{split}
        &\sum_{n=1}^{\infty}\P^S_0((\log a_n)^{1/3}\P^W_0(W[1,a_n]\cap S[a_n+1,\infty)\neq\varnothing\mid S)>\varepsilon)\\
        \lesssim&\varepsilon\sum_{n=1}^\infty(\log a_n)^{-2/3}\leq\varepsilon\sum_{n=1}^\infty n^{-4/3}<\infty.
    \end{split}
\end{equation}
So by the Borel-Cantelli lemma, $\limn(\log a_n)^{1/3}\P^W_0(W[1,a_n]\cap S[a_n+1,\infty)\neq\varnothing\mid S)=0$ almost surely. So almost surely,
\begin{equation}
\limsup_{n\rightarrow\infty}\widetilde{X}_{a_n}^+\leq\limsup_{n\rightarrow\infty}X_{a_n}^+\leq X_\infty^+.
\end{equation}
Conversely,
\begin{equation}
    \begin{split}
        1-\widetilde{X}^+_n\leq&(\log n)^{1/3}\P_0^W(W[1,n]\cap\eta[0,n]\neq\varnothing\mid\eta)\\
        &+(\log n)^{1/3}\P^W_0(W[n+1,\infty)\cap\eta[0,n]\neq\varnothing\mid\eta)\\
        \leq&1-X^+_n+(\log n)^{1/3}\P^W_0(W[n+1,\infty)\cap\eta[0,n]\neq\varnothing\mid\eta).
    \end{split}
\end{equation}
Recall from Lemma \ref{lemma:cuttimes} that $\E[\l_n]\sim n(\log n)^{1/3}$. So by Markov's inequality, 
\begin{equation}
	\P(\eta[0,n]\not\subset S[0,n(\log n)^2])\lesssim(\log n)^{-5/3}.
\end{equation}
Furthermore, it follows from \cite[Theorem 1.1]{albeverio96intersections} that 
\begin{equation}
	\E^S_0[\P^W_0(W[n+1,\infty)\cap S[0,n(\log n)^2]\neq\varnothing\mid S)]\lesssim\frac{\log\log n}{\log n}.
\end{equation}
Combining this, we have 
\begin{equation}
	\E^S_0[\P^W_0(W[n+1,\infty)\cap\eta[0,n]\neq\varnothing\mid\eta)]\lesssim\frac{\log\log n}{\log n}.
\end{equation}
So by Markov's inequality,
\begin{equation}
    \begin{split}
        \sum_{n=1}^{\infty}\P_0^S((\log a_n)^{1/3}\P^W_0(W[a_n+1,\infty)\cap \eta[0,a_n]\neq\varnothing\mid \eta)>\varepsilon)\lesssim\sum_{n=1}^\infty\frac{\log\log a_n}{(\log a_n)^{2/3}}<\infty.
    \end{split}
\end{equation}
Another application of Borel-Cantelli gives that 
\begin{equation}
\liminf_{n\rightarrow\infty}\widetilde{X}_{a_n}^+\geq\liminf_{n\rightarrow\infty}X_{a_n}^+=X_\infty^+
\end{equation}
almost surely, and so,
\begin{equation}
\lim_{n\rightarrow\infty}\widetilde{X}_{a_n}^+=X_\infty^+.
\end{equation}
Convergence of the full sequence $\widetilde{X}_n^+$ now follows by applying Lemma \ref{lem:subseqconv} with $f(n)=\P_0^W(W[1,\infty)\cap\eta[0,n]=\varnothing\mid\eta)$ and $b_n=(\log n)^{1/3}$.

We now show convergence in $L^p$ for $1\leq p<3$. Let $\|\cdot\|_p$ denote the $L^p$-norm of a random variable. From the above, we obtain
\begin{equation}\label{eq:lpconvergenceXntilde}
    \begin{split}
        \|X^+_\infty-\widetilde{X}^+_n\|_p\leq&\|X^+_\infty-X^+_n\|_p+\|X^+_n-\widetilde{X}_n\|_p\\
        \leq&\|X^+_\infty-X^+_n\|_p+(\log n)^{1/3}\left(\E^\eta_0\left[\P^W_0(W[1,n]\cap \eta[n+1,\infty)\neq\varnothing\mid\eta)^p\right]\right)^{1/p}\\
        &+(\log n)^{1/3}\left(\E^\eta_0\left[\P^W_0(W[n+1,\infty)\cap \eta[0,n]\neq\varnothing\mid\eta)^p\right]\right)^{1/p}.\\
        \leq&\|X^+_\infty-X^+_n\|_p+(\log n)^{1/3}\left(\E^S_0\left[\P^W_0(W[1,n]\cap S[n+1,\infty)\neq\varnothing\mid\eta)\right]\right)^{1/p}\\
        &+(\log n)^{1/3}\left(\E^S_0\left[\P^W_0(W[n+1,\infty)\cap \eta[0,n]\neq\varnothing\mid\eta)\right]\right)^{1/p}.\\
        \lesssim&\|X^+_\infty-X^+_n\|_p+\frac{(\log n)^{1/3}}{(\log n)^{1/p}}+\frac{(\log\log n)^{1/p}(\log n)^{1/3}}{(\log n)^{1/p}}\rightarrow0,\qquad n\rightarrow\infty,
    \end{split}
\end{equation}
which settles the proof.
\end{proof}

Recall that we want to prove a strong law of large numbers for $\Ca(\heta[0,n])$, rather than $\eta[0,n]$. In the following lemma, we show that the corresponding escape probabilities also converge for the two-sided LERW.
\begin{lemma}\label{lem:Xconvergencetwosided}
There exist random variables $\hat{X}_\infty$ and $\hat{X}^+_\infty$ depending on $\heta$ such that
\begin{equation}
    \begin{split}
        \hat{X}^+_\infty&=\limn \hat{X}^+_n:=\limn(\log n)^{1/3}\P^W_0(W[1,\infty)\cap\hat{\eta}[1,n]=\varnothing\mid \heta),\\
        \hat{X}_\infty&=\limn \hat{X}_n:=\limn(\log n)^{1/3}\P^W_0(W[1,\infty)\cap\hat{\eta}[0,n]=\varnothing\mid\heta)
    \end{split}
\end{equation}
almost surely and in $L^p$ for every $p<3$.
\end{lemma}

\begin{proof}
Almost sure convergence follows from absolute continuity of the law of $\hat{\eta}$ combined with Lemma \ref{lem:convergenceXtilde}. Furthermore, let $f$ be the Radon-Nikodym derivative of $\heta[0,\infty)$ with respect to $\eta$. Let $p<3$. Let $p'=\frac{3+p}{2}$ and let $q'$ be the H\"older conjugate of $p'$. Then by H\"older's inequality,
\begin{equation}
	\begin{split}
		\|\hat{X}_\infty^+-\hat{X}_n^+\|^p_p=&\E^{\eta}[f|\widetilde{X}_\infty^+(\eta)-\widetilde{X}_n^x(\eta)|^p]\\
		\leq&\E^{\eta}\left[f^{\frac{q'}{p}}\right]^{\frac{p}{q'}}\E^{\eta}\left[|\hat{X}_{\infty}^+-\hat{X}_{n}^+|^{p'}\right]^{\frac{p}{p'}}\rightarrow0,\qquad n\rightarrow\infty.
	\end{split}
\end{equation}
For the final step, we use that $f$ is $L^a$ bounded for all $a>0$ and Lemma \ref{lem:convergenceXtilde} in combination with the fact that $p'<3$.
\end{proof}

\subsubsection{Lower bound}
Let $\varepsilon>0$. By Lemma \ref{lem:Xconvergencetwosided}, there exists $n_0=n_0(\heta)$ such that for all $n\geq n_0$, $X_\infty\leq\widetilde{X}_n+\varepsilon$ and $X^+_\infty\leq\widetilde{X}^+_n+\varepsilon$. Then for all $n\geq m$,
\begin{equation}
    \begin{split}
        &\liminf_{n\rightarrow\infty}\frac{(\log n)^{2/3}}{n}\Ca(\heta[0,n])\\
      	=&\liminf_{n\rightarrow\infty}\frac{(\log n)^{2/3}}{n}\sum_{k=0}^n\P^W_{\eta(k)}(W[1,\infty)\cap\eta[k,n]=\varnothing\mid\eta)\P^W_{\eta(k)}(W[1,\infty)\cap\eta[k+1,n]=\varnothing\mid\eta)\\
        \geq&\liminf_{n\rightarrow\infty}\frac{1}{n}\sum_{k=0}^n\hat{X}_n(T^k(\heta))\hat{X}_n^+(T^k(\heta))\\
        \geq&\liminf_{n\rightarrow\infty}\frac{1}{n}\sum_{k=0}^n\hat{X}_\infty(T^k(\heta))\hat{X}_\infty^+(T^k(\heta))\1_{\{m\geq n_0\}}(T^k(\heta))-\varepsilon\\
        =&\E[\hat X_\infty \hat X_{\infty}^+\1_{\{m\geq n_0\}}]-\varepsilon\\
        \rightarrow&\E[\hat X_\infty\hat X^+_\infty]-\varepsilon,\qquad m\rightarrow\infty
    \end{split}
\end{equation}
The first line follows from Lemma \ref{lem:capacitydecomp}, the fourth line follows from Birkhoff's ergodic theorem \cite[Theorem 24.1]{billingsley1995probability}, and the last line follows from the monotone convergence theorem. Since $\varepsilon>0$ was arbitrary, we conclude that
\begin{equation}
	\liminf_{n\rightarrow\infty}\frac{(\log n)^{2/3}}{n}\Ca(\heta[0,n])\geq\E[\hat X_\infty\hat X^+_\infty].
\end{equation}

\subsubsection{Upper bound}
We show an upper bound for $\frac{1}{n}\Ca(\heta[0,\lfloor n(\log n)^{2/3}\rfloor])$.
The strategy is similar to the upper bound in high dimensions. Let $m\in\N$. By subadditivity of capacity,
\begin{equation}
	\begin{split}
		\frac{(\log n)^{2/3}}{n}\Ca(\heta[0,n])\leq&\limsup_{n\rightarrow\infty}\sum_{k=0}^{\lceil\frac{n(\log n)^{-2/3}}{m(\log m)^{-2/3}}\rceil-1}\Ca\left(\heta[mk,m(k+1)]\right)\\
		\leq&\frac{(\log m)^{2/3}}{m}\frac{1}{\frac{n(\log n)^{-2/3}}{m(\log m)^{-2/3}}}\sum_{k=0}^{\left\lceil\frac{n(\log n)^{-2/3}}{m(\log m)^{-2/3}}\right\rceil-1}\Ca(((T^m)^k\heta)[0,m])
	\end{split}
\end{equation}
By the exact same proof as the proof of Proposition \ref{prop:ergodicity4}, it follows that every power of $T$ is also ergodic. So by Birkhoff's ergodic theorem,
\begin{equation}
	\limsup_{n\rightarrow\infty}\frac{(\log n)^{2/3}}{n}\Ca(\heta[0,n])\leq\frac{(\log m)^{2/3}}{m}\E[\Ca(\heta[0,m])]
\end{equation}
almost surely for all $m$. It remains to show that
\begin{equation}
    \limsup_{m\rightarrow\infty}\frac{(\log m)^{2/3}}{m}\E[\Ca(\heta[0,m])]\leq\E[\hat X_\infty \hat X^+_\infty].
\end{equation}

Let $W^1,W^2$ be simple random walks and let $\varepsilon>0$. Then
\begin{equation}
    \begin{split}
        &\frac{(\log m)^{2/3}}{m}\E[\Ca(\heta[0,m])]\\
        =&\frac{(\log m)^{2/3}}{m}\sum_{k=0}^m\P(W^1[1,\infty)\cap\heta[0,k]=\varnothing,\,W^2[1,\infty)\cap\heta[1,k]\mid W^1_0=W^2_0=\heta(k))\\
        =&\frac{1}{m}\sum_{k=0}^m\frac{(\log m)^{2/3}}{(\log k)^{2/3}}\E[\hat{X}_{k}\hat{X}_{k}^+]\\
        \leq&\frac{1}{m}\sum_{k=\lfloor n^{1-\varepsilon}\rfloor}^m\frac{(\log m)^{1/3}}{(\log m^{1-\varepsilon})^{1/3}}\E[\hat{X}_{k}\hat{X}_{k}^+] + \frac{1}{m}m^{1-\varepsilon}(\log m)^{1/3}\\
        \leq&\frac{1}{(1-\varepsilon)^{1/3}m}\sum_{k=0}^m\E[\hat{X}_{k}\hat{X}_{k}^+]+o(1)\\
        \rightarrow&(1-\varepsilon)^{-1/3}\E[\hat{X}_{\infty}\hat{X}_{\infty}^+],\qquad m\rightarrow\infty.
    \end{split}
\end{equation}
The convergence in the last line follows from Lemma \ref{lem:Xconvergencetwosided}. Since $\varepsilon>0$ was arbitrary, the proof is complete.

\section{Capacity of LERW in three dimensions}\label{sec:d3proof}
This section is devoted to the proof of Theorem \ref{th:d3}. We first prove a technical hitting lemma in Section \ref{sec:hittinglemma}. In Section \ref{sec:d3proofsub}, we  prove Theorem \ref{th:d3}.
\subsection{A hitting estimate}\label{sec:hittinglemma}
As mentioned in Section \ref{sec:sketch-lerw}, for the proof we need that a SRW started close to $\eta[0,n]$ will hit it with high probability. There are several known results in this direction (most notably \cite[Lemma 3.3]{sapozhnikov18brownian} and \cite[Proposition 3.7]{angel21scaling}), but none match the exact statement we need.

\begin{lemma}\label{lemma:hitting3dLERW}
There exist $C,\alpha>0$ such that for all $\varepsilon>0$, there exists $\delta_0$ such that for all $\delta\leq\delta_0$ and for all $n$,
\begin{equation}
    \P\left(\max_{z\in B(\eta[0,n],\delta n^{1/\beta})}\P^W_z(W[0,\infty)\cap \eta[0,n]=\varnothing\mid \eta)>\varepsilon\right)< C\varepsilon^{1/\alpha}.
\end{equation}
\end{lemma}
For the proof, we reformulate the hitting estimate we need in terms of what was shown in \cite[Lemma 3.3]{sapozhnikov18brownian} and \cite[Proposition 3.7]{angel21scaling}. This has as advantage that we can use these results as a `black box', at the cost of making the proof somewhat opaque.

\begin{proof}[Proof of Lemma \ref{lemma:hitting3dLERW}]
Let $\alpha>0$ be some constant, whose exact value we will specify later. Let $s=\delta n^{1/\beta}$ and $r=\varepsilon^{-1/\alpha}s$. Then
\begin{equation}\label{eq:splitball}
    \begin{split}
        &\P^\eta\left(\max_{z\in B(\eta[0,n],\delta n^{1/\beta})}\P^W_z(W[0,\infty)\cap \eta[0,n]=\varnothing\mid \eta)>\varepsilon\right)\\
        \leq&\sum_{z\not\in B(0,r)}\P^\eta\left(\left\{\eta[0,n]\cap B(z,s)\neq\varnothing\right\}\cap\left\{\P^W_z(W[0,\infty)\cap \eta[0,n]=\varnothing\mid \eta)>\varepsilon\right\}\right)\\
        &+\P^\eta\left(\max_{z\in B(0,r)}\P^W_z(W[0,\infty)\cap \eta[0,n]=\varnothing\mid \eta)>\varepsilon\right).
    \end{split}
\end{equation}
We bound both terms separately. Let $\tau^\eta_A:=\min\{t\geq0\colon \eta(t)\in A\}$. If $z\in B(\eta[0,n],s)$, then $\tau^\eta_{B(z,s)}\leq n$. Furthermore, for $\|z\|\geq r$, we have $(\frac{s}{\|z\|})^{\alpha}\leq(\frac{s}{r})^{\alpha}=\varepsilon$. Hence, by \cite[Lemma 3.3]{sapozhnikov18brownian}, there exists $\alpha$ and a constant $C_1$ such that for all $n$,
\begin{equation}\label{eq:firstterm}
    \begin{split}
        &\sum_{z\not\in B(0,r)}\P^\eta\left(\eta[0,n]\cap B(z,s)\neq\varnothing,\,\P^W_z(W[0,\infty)\cap \eta[0,n]=\varnothing\mid \eta)>\varepsilon\right)\\
        \leq&\sum_{z\not\in B(0,r)}\P^\eta\left(\eta[0,\infty)\cap B(z,s)\neq\varnothing,\,\P^W_z(W[0,\infty)\cap \eta[0,\tau^{\eta}_{B(z,s)}]=\varnothing\mid \eta)>(\tfrac{s}{\|z\|})^{\alpha}\right)\\
        \leq&\sum_{z\not\in B(0,r)} C_1\left(\frac{s}{\|z\|}\right)^{4}\\
        \lesssim&\frac{s}{r}=\varepsilon^{1/\alpha}.
    \end{split}
\end{equation}
We now bound the second term of \eqref{eq:splitball}. Let $\hat{\alpha}$ be the constant $\hat\eta$ from \cite[Proposition 3.7]{angel21scaling}. Then there exists $C_2$ such that for all $\delta\leq\varepsilon^{1/\hat\alpha}$,
\begin{equation}
    \begin{split}
        &\P^{\eta}\left(\max_{z\in B(0,r)}\P_z^W(W[0,\infty)\cap\eta[0,n]=\varnothing)>\varepsilon\right)\\
        \leq&\P^{\eta}\left(\left\{\max_{z\in B(0,r)}\P_z^W(W[0,\infty)\cap\eta[0,n]=\varnothing)>\delta^{\hat\alpha}\right\}\right.\\
        &\left.\cap\left\{\eta(n,\infty)\cap B(0,2\delta^{1/2}\varepsilon^{-1/\alpha} n^{1/\beta})=\varnothing\right\}\right)\\
        &+\P^{\eta}\left(\eta(n,\infty)\cap B(0,2\delta^{1/2}\varepsilon^{-1/\alpha} n^{1/\beta})\neq\varnothing\right)\\
        \leq&\P^{\eta}\left(\max_{z\in B(0,\delta\varepsilon^{-1/\alpha}n^{1/\beta})}\P_z^W(W[0,\tau^W_{B(z,\delta^{1/2}\varepsilon^{-1/\alpha}n^{1/\beta})^c}]\cap\eta[0,\infty)=\varnothing)>\delta^{\hat\alpha}\right)\\
        &+\P^{\eta}\left(\eta(n,\infty)\cap B(0,2\delta^{1/2}\varepsilon^{-1/\alpha} n^{1/\beta})\neq\varnothing\right)\\
        \leq&C_2\delta+\P^{\eta}\left(\eta(n,\infty)\cap B(0,2\delta^{1/2}\varepsilon^{-1/\alpha} n^{1/\beta})\neq\varnothing\right).
    \end{split}
\end{equation}
We bound the final probability. Note that
\begin{equation}
    \begin{split}
        \P^{\eta}\left(\eta(n,\infty)\cap B(0,2\delta^{1/2}\varepsilon^{-1/\alpha} n^{1/\beta})\neq\varnothing\right)=&\P^{\eta}\left(n^{-1/\beta}\eta(n,\infty)\cap B(0,2\delta^{1/2}\varepsilon^{-1/\alpha})\neq\varnothing\right)
    \end{split}
\end{equation}
By convergence of $n^{-1/\beta}\eta(n,\infty)$ to the scaling limit $\gamma(1,\infty)$, the final probability can be made arbitrarily small by choosing $\delta$ sufficiently small. Hence, there exists $\delta_0$ such that for all $\delta\leq\delta_0$, and all $n$,
\begin{equation}\label{eq:secondterm}
    \begin{split}
        \P^{\eta}\left(\max_{z\in B(0,r)}\P_z^W(W[0,\infty)\cap\eta[0,n]=\varnothing)>\varepsilon\right)\leq C_3\varepsilon^{1/\alpha}.
    \end{split}
\end{equation}
We conclude the proof by combining \eqref{eq:firstterm} and \eqref{eq:secondterm}.
\end{proof}

\subsection{Proof of Theorem \ref{th:d3}}\label{sec:d3proofsub}
We now prove Theorem \ref{th:d3}. In fact, we prove the following stronger statement.
\begin{proposition}\label{prop:d3-general}
Let $\ell\in\N$ and let $0\leq u_1\leq\ldots\leq u_{\ell}$. There exists a coupling of $(n^{-1/\beta}\eta(n\cdot))_{n\in\N}$ and $\gamma$ such that $(n^{-1/\beta}\eta(n\cdot))_{n\in\N}$ converges almost surely to $\gamma$ in $(\mathcal{C},\chi)$ and for all $u_i$,
\begin{equation}
    \limn\frac{\Cz(\eta[0,\lfloor nu_i\rfloor]}{3n^{1/\beta}}=\Cr(\gamma[0,u_i]).
\end{equation}
\end{proposition}
The proof is divided into three main steps, plus a preparation step. In each of the main steps we give a rigorous formulation and proof of the approximate equalities \eqref{eq:step1sketch}--\eqref{eq:step3sketch}.

\begin{proof}[Proof of Proposition \ref{prop:d3-general}]
Let $u_1,\ldots,u_{\ell}$. We first show that we can couple $(n^{-1/\beta}\eta(n\cdot))_{n\in\N}$ and $\gamma$ such that for all $u\in\{u_1,\ldots,u_\ell\}$,
\begin{equation}\label{eq:lln3dupperbound}
    \frac{1}{n^{1/\beta}}\Ca_{\Z^3}(\eta[0,\lfloor n u\rfloor])\leq3\Ca_{\R^3}(\gamma[0,u])+o(1),\qquad n\rightarrow\infty.
\end{equation}
The proof of the reverse inequality is almost analogous and we briefly remark on it at the end.

\paragraph{Step 0.}
We first show that with high probability, $\eta$ and its scaling limit $\gamma$ are well-behaved. Let $\varepsilon>0$ and let $\delta>0$. We will later choose $\delta>0$ to be small depending on $\varepsilon$ and $u_{\ell}$ and then $n$ to be sufficiently large depending on $\delta,\varepsilon,u_{\ell}$. Let $W$ be an independent simple random walk on $\Z^d$ and $M$ be an independent Brownian motion on $\R^d$. Define the events
\begin{equation}
    \begin{split}
        E_1(n,\delta,u_{\ell}):=&\left\{\sup_{0\leq t\leq u_{\ell}}\left\|\frac{1}{n^{1/\beta}}\eta(\lfloor n t\rfloor)-\gamma(t)\right\|\leq\delta\right\},\\
        E_2(n,\delta,\varepsilon,u_1,\ldots,u_{\ell}):=&\left\{\max_{1\leq i\leq \ell}\sup_{z\in B(\eta[0,\lfloor nu_i\rfloor],\delta n^{1/\beta})}\P^W_z(W[0,\infty)\cap \eta[0,\lfloor nu_i\rfloor]=\varnothing\mid \eta)<\varepsilon\right\},\\
        E_3(\delta,\varepsilon,u_1,\ldots,u_{\ell}):=&\left\{\max_{1\leq i\leq\ell}\sup_{z\in B(\gamma[0,u_i],\delta)}\P^M_z(M[0,\infty)\cap \gamma[0,u_i]=\varnothing\mid \gamma)<\varepsilon\right\},\\
        E_4(\delta,u_{\ell}):=&\left\{\sup_{0\leq t\leq u_{\ell}}\|\gamma(t)\|\geq\delta^{-1}\right\},\\
        E(n,\delta,\varepsilon,u_1,\ldots,u_{\ell}):=&E_1\cap E_2\cap E_3\cap E_4.
    \end{split}
\end{equation}
By convergence of $n^{-1/\beta}\eta(n\cdot)$ to $\gamma$, for all $\delta>0$ and $n$ sufficiently large depending on $\delta$, we have $\P(E_1)>1-\varepsilon$. Furthermore, by Lemma \ref{lemma:hitting3dLERW}, there exist constants $C,\alpha>0$ such that for all $\delta$ sufficiently small and all $n>0$, $\P(E_2)>1-C\ell\varepsilon^{1/\alpha}$. Moreover, since the Hausdorff dimension of $\gamma[0,u_i]$ is strictly greater than 1, $\sup_{z\in B(\gamma[0,u_i],\delta)}\P^M_z(M[0,\infty)\cap \gamma[0,u_i]=\varnothing\mid\gamma)$ tends to 0 as $\delta\downarrow 0$ almost surely. So for $\delta$ sufficiently small, $\P(E_3)>1-\varepsilon$. Lastly, since $\gamma[0,u_{\ell}]$ is bounded almost surely, $\P(E_4)>1-\varepsilon$ for $\delta$ sufficiently small. Combining the above, we obtain that for all $\varepsilon>0$, for all $\delta>0$ sufficiently small and for all $n$ sufficiently large, we have $\P(E(n,\delta,\varepsilon,u_i))\geq1-\varepsilon$. Thus, we can couple $(n^{-1/\beta}\eta(n\cdot))_{n\in\N}$ and $\gamma$ such that $E(n,\delta_n,\varepsilon_n,u_i)$ holds for all $n$ for some sequences $(\varepsilon_n)_{n\in\N},(\delta_n)_{n\in\N}$ satisfying $\varepsilon_n,\delta_n\downarrow0$ as $n\rightarrow\infty$. Now let $u\in\{u_1,\ldots,u_{\ell}\}$. We show that \eqref{eq:lln3dupperbound} holds under this coupling.

\paragraph{Step 1.}
We first show that under the coupling above, we can approximate the capacity of $\eta[0,\lfloor nu\rfloor]$ by the capacity of a sufficiently small sausage around it. Let $\varepsilon,\delta>0$. Let $A\subset B\subset\Z^d$ and denote by $h(B)$ the harmonic measure on $B$. It is not difficult to prove that
\begin{equation}\label{eq:capacitysubsetZd}
    \Ca_{\Z^3}(A)=\Ca_{\Z^3}(B)\P_{h(B)}^W(W[0,\infty)\cap A\neq\varnothing).
\end{equation}
Similarly, for $A\subset B\subset\R^d$, we have
\begin{equation}\label{eq:capacitysubsetRd}
    \Ca_{\R^3}(A)=\Ca_{\R^3}(B)\P_{h(B)}^M(M[0,\infty)\cap A\neq\varnothing).
\end{equation}
Note that on the event $E\subset E_2$, for $n$ sufficiently large, we have
\begin{equation}
    \P_{h(B(\eta[0,\lfloor nu\rfloor],3\delta n^{1/\beta}))}^W(W[0,\infty)\cap \eta[0,\lfloor nu\rfloor]\neq\varnothing)\geq 1-\varepsilon.
\end{equation}
Thus, 
\begin{equation}
    \frac{1}{n^{1/\beta}}\Ca_{\Z^3}(\eta[0,\lfloor nu\rfloor])\geq (1-\varepsilon)\frac{1}{n^{1/\beta}}\Ca_{\Z^3}(B(\eta[0,\lfloor nu\rfloor],3\delta n^{1/\beta})).
\end{equation}

\paragraph{Step 2.}
We now show that $\Ca_{\Z^3}(B(\eta[0,\lfloor nu\rfloor],3\delta n^{1/\beta}))\approx3\Ca_{\R^3}(B(\eta[0,\lfloor nu\rfloor],3\delta n^{1/\beta}))$. We do this by showing that the probability that $B(\eta[0,\lfloor nu\rfloor],3\delta n^{1/\beta})$ is hit by a random walk on $\Z^3$ is close to the probability that $B(\eta[0,\lfloor nu\rfloor],3\delta n^{1/\beta})$ is hit by Brownian motion on $\R^3$. For this part, we need several standard estimates on the hitting probabilities and displacement of simple random walk and Brownian motion, which we will state without proof.

By \cite[Theorem 2]{einmahl1989extensions}, for a SRW $W$ and Brownian motion $M$ started from the same point there exists a coupling of $W$ and $M$ such that almost surely,
\begin{equation}
        \sup_{0\leq k\leq e^{m^{1/4}}}\|W_{k}-M_{k/3}\|\leq\sqrt{m}
\end{equation}
for every $m$. This is a very strong coupling, which is somewhat overkill for our purposes. Recall that on the event $E$, $n^{1/\beta}\gamma[0,u]\subset B(0,2n^{1/\beta}\delta_n^{-1})$. Now consider $W$ and $M$ started from the harmonic measures on $B(0,2n^{1/\beta}\delta^{-1})$ in $\Z^3$ and $\R^3$ respectively. Since the harmonic measure on $B(0,2\delta^{-1})$ in $n^{-1/\beta}\Z^3$ converges weakly to the harmonic measure on $B(0,2\delta^{-1})$ in $\R^3$ as $n\rightarrow\infty$, we can couple $W$ and $M$ such that $\|W_0-M_0\|\leq\frac{1}{4}\delta n^{1/\beta}$ and thus
\begin{equation}
    \begin{split}
        \sup_{0\leq k\leq n^{100/\beta}}|W_{k}-M_{k/3}|\leq\frac{2}{4}\delta n^{1/\beta}.
    \end{split}
\end{equation}
Furthermore, it is not difficult to show that 
\begin{align}
    \P_{h(B(0,2\delta^{-1}n^{1/\beta}))}(W(\lfloor n^{100/\beta}\rfloor,\infty)\cap B(0,2\delta^{-1}n^{1/\beta}))&\rightarrow0,\\
    \P_{h(B(0,2\delta^{-1}n^{1/\beta}))}(M(\lfloor n^{100/\beta}/3\rfloor,\infty)\cap B(0,2\delta^{-1}n^{1/\beta}))&\rightarrow0
\end{align}
as $n\rightarrow\infty$. Also, in between integer times, with high probability, the Brownian motion does not stray too far:
\begin{equation}
    \begin{split}
        \P\left(\exists0\leq k\leq n^{100/\beta}\colon \sup_{k\leq t\leq k+1}\|M_t-M_k\|\geq\frac{1}{4}\delta n^{1/\beta}\right)\rightarrow0,\qquad n\rightarrow\infty.
    \end{split}
\end{equation}
Combining the above, we obtain that on the event $E$, for $n$ sufficiently large depending on $\delta$, 
\begin{equation}
    \begin{split}
        &\P^W_{h(B(0,2\delta^{-1}n^{1/\beta}))}(W[0,\infty)\in B(\eta[0,\lfloor nu\rfloor],3\delta n^{1/\beta}))\\
        \geq&\P^M_{h(B(0,2\delta^{-1}n^{1/\beta}))}(M[0,\infty)\in B(\eta[0,\lfloor nu\rfloor],2\delta n^{1/\beta})).
    \end{split}
\end{equation}
Thus, by \eqref{eq:capacitysubsetZd} and \eqref{eq:capacitysubsetRd}, we have
\begin{equation}
    \begin{split}
        \Ca_{\Z^3}(B(\eta[0,\lfloor nu\rfloor],3\delta n^{1/\beta})\geq 3\Ca_{\R^3}(B(\eta[0,\lfloor nu\rfloor],2\delta n^{1/\beta})).
    \end{split}
\end{equation}
The extra factor 3 comes from the fact that $\lim_{\|y\|\rightarrow\infty}\frac{G_{\Z^d}(0,y)}{G_{\R^d}(0,y)}=3$, see \cite[Theorem 3.3]{morters2010brownian} and \cite[Theorem 4.3.1]{lawler2010random}.

\paragraph{Step 3.}
On the event $E$, we have $B(n^{-1/\beta}\eta[0,\lfloor nu\rfloor],2\delta)\supset B(\gamma[0,u],\delta)\supset\gamma[0,u]$, so
\begin{equation}
    3\Ca_{\R^3}(B(n^{-1/\beta}\eta[0,\lfloor nu\rfloor],2\delta))\geq3\Ca_{\R^3}(\gamma[0,u]).
\end{equation}
Combining Steps 1-3, we obtain that for all $\varepsilon>0$, there exists $n_0$ such that for all $n\geq n_0$,
\begin{equation}
    \frac{1}{n^{1/\beta}}\Ca_{\Z^3}(\eta[0,\lfloor nu\rfloor])\geq3(1-\varepsilon)\Cr(\gamma[0,u]),
\end{equation}
which concludes the proof of \eqref{eq:lln3dupperbound}.

\paragraph{Reverse inequality.} The reverse inequality of \eqref{eq:lln3dupperbound} is entirely analogous. Steps 0, 2 and 3 are the same. Step 1 is also the same, only using that $E\subset E_3$, instead of $E_2$.
\end{proof}

\subsection{Proof of Theorem \ref{th:scaling-limit-cap-par}}
Let $C_n(u)=n^{-1}\Cz(\eta[0,\lfloor 3^{-\beta}n^{\beta}u\rfloor])$ and $C(u)=\Cr(\gamma[0,u])$. Let $D([0,\infty),[0,\infty))$ be the space of c\`adl\`ag processes in $[0,\infty)$ equipped with the metric $\chi$ as defined in Section \ref{sec:d3res}
and consider the right-inverse function $F\colon D([0,\infty),[0,\infty))\to D([0,\infty),[0,\infty))$, defined
\begin{equation}
    F(A)(t):=\inf\{v:A(v)> t\}.
\end{equation}
Then $\tau_n=3^{-\beta}n^{\beta}F(C_n)$ and $\theta=F(C)$. 

\begin{lemma}\label{lem:convergence-Cn}
There exists a coupling of $(n^{-1/\beta}\eta(n\cdot))_{n\in\N}$ and $\gamma$ such that $n^{-1/\beta}\eta(n\cdot)$ converges almost surely to $\gamma$ in $(\mathcal{C},\chi)$ and such that $(C_n)_{n\in\N}$ converges almost surely to $C$ in $(D([0,\infty),[0,\infty)),\chi)$ as $n\rightarrow\infty$.
\end{lemma}
\begin{proof}
First note that for $u_1<u_2$, $C_n(u_2)-C_n(u_1)$ is bounded from above by a universal constant times the diameter of $n^{-1}\eta[\lfloor3^{-\beta}n^\beta u_1\rfloor,\lfloor 3^{-\beta}n^\beta u_2\rfloor])$. Convergence of $(n^{-1}\eta(3^{-\beta}n^\beta t))_{t\geq0}$ with respect to the metric $\chi$ combined with the Arzel\`a-Ascoli theorem thus implies that $C_n$ is sequentially compact with respect to $\chi$. It follows from convergence of the finite-dimensional distributions in Proposition~\ref{prop:d3-general} that all subsequential limits must equal $C$.
\end{proof}

\begin{proof}[Proof of Theorem \ref{th:scaling-limit-cap-par}]
Couple $\eta$ and $\gamma$ as in Lemma~\ref{lem:convergence-Cn}. Then  $C_n\rightarrow C$ almost surely with respect to $\chi$ as $n\rightarrow\infty$ by Lemma \ref{lem:convergence-Cn}. Note that $\gamma$ is a.s. continuous and injective \cite[Section 9.1]{li2025convergence}. This implies that $C$ is a.s. continuous and strictly increasing. It follows from \cite[Corollary 13.6.4]{whitt2002stochastic} that $C$ is a.s. a continuity point of $F$. The continuous mapping theorem then implies that $F(C_n)$ converges a.s. to $F(C)$. So almost surely, for all $s\geq0$,
\begin{equation}
    3^{\beta}n^{-\beta}\tau_{n}(s)=F(C_{n})(s)\rightarrow F(C)(s)=\theta(s)
\end{equation}
as $n\rightarrow\infty$. Let $T>0$. By convergence of $(n^{\beta}\eta(n\cdot))_{n\in\N}$ and by the diameter bound on capacity, there exists $U$ such that $\tau_n(T)\leq n^\beta U$ for all $n\in\N$. Then
\begin{equation}
    \begin{split}
        &\,\sup_{0\leq t\leq T}|\widetilde{\eta}_n(t)-\widetilde\gamma(t)|\\
        =&\,\sup_{0\leq t\leq T}|n^{-1}\eta_n(\tau_n(t))-\gamma(3^{-\beta}\theta(t))|\\
        \leq&\,\sup_{0\leq t\leq T}|n^{-1}\eta_n(\tau_n(t))-\gamma(3^{-\beta}3^{\beta}n^{-\beta}\tau_n(t))|+|\gamma(3^{-\beta}3^{\beta}n^{-\beta}\tau_n(t))-\gamma(3^{-\beta}\theta(t))|\\
        \leq&\,\sup_{0\leq u\leq U}|n^{-1}\eta([0,n^{\beta}u])-\gamma(u)|+\sup_{0\leq t\leq T}|\gamma(3^{-\beta}3^{\beta}n^{-\beta}\tau_n(t))-\gamma(3^{-\beta}\theta(t))|.
    \end{split}
\end{equation}
The first term tends to 0 by convergence of $n^{-1}\eta(n^{\beta}\cdot)$ to $\gamma$ with respect to the metric $\chi$, which is equivalent to uniform convergence on compact domains. To see that the second term tends to 0, we use that $F(C_{n})$ converges to $\theta$ with respect to $\chi$, which implies that $F(C_n)$ converges uniformly to $\theta$ on $[0,T]$. Then by uniform continuity of $\gamma$ on $[0,U]$, we conclude that the second term vanishes, which completes the proof.
\end{proof}

\paragraph{Acknowledgements} The author is very grateful to Perla Sousi for her comments on this paper and the many mathematical discussions. This work was supported by the University of Cambridge Harding Distinguished Postgraduate Scholarship Programme

\bibliographystyle{hmralphaabbrv}
\bibliography{referenties}

\begin{thebibliography}{ACHTS21}

\bibitem[ACHTS21]{angel21scaling}
O.~Angel, D.~A. Croydon, S.~Hernandez-Torres, and D.~Shiraishi.
\newblock Scaling limits of the three-dimensional uniform spanning tree and
  associated random walk.
\newblock {\em Ann. Probab.}, 49(6):3032--3105, 2021. \MR{4348685}

\bibitem[ANS24]{archer24ghp}
E.~Archer, A.~Nachmias, and M.~Shalev.
\newblock The {GHP} scaling limit of uniform spanning trees in high dimensions.
\newblock {\em Comm. Math. Phys.}, 405(3):Paper No. 73, 41, 2024. \MR{4712854}

\bibitem[AS24]{archer24dense}
E.~Archer and M.~Shalev.
\newblock The {GHP} scaling limit of uniform spanning trees of dense graphs.
\newblock {\em Random Structures Algorithms}, 65(1):149--190, 2024.
  \MR{4760770}

\bibitem[ASS18]{asselah18capacity}
A.~Asselah, B.~Schapira, and P.~Sousi.
\newblock Capacity of the range of random walk on $\mathbb{Z}^d$.
\newblock {\em Transactions of the American Mathematical Society},
  370(11):7627--7645, 2018.

\bibitem[ASS19]{asselahcapacity19}
A.~Asselah, B.~Schapira, and P.~Sousi.
\newblock Capacity of the range of random walk on {$\mathbb{Z}^4$}.
\newblock {\em Ann. Probab.}, 47(3):1447--1497, 2019. \MR{3945751}

\bibitem[Ass23]{asselahprivate}
A.~Asselah.
\newblock Private communication, 2023.

\bibitem[AZ96]{albeverio96intersections}
S.~Albeverio and X.~Y. Zhou.
\newblock Intersections of random walks and {W}iener sausages in four
  dimensions.
\newblock {\em Acta Appl. Math.}, 45(2):195--237, 1996. \MR{1414282}

\bibitem[Bil95]{billingsley1995probability}
P.~Billingsley.
\newblock {\em Probability and measure}.
\newblock Wiley Series in Probability and Mathematical Statistics. John Wiley
  \& Sons, Inc., New York, third edition, 1995.
\newblock A Wiley-Interscience Publication. \MR{1324786}

\bibitem[Cha17]{chang17two}
Y.~Chang.
\newblock Two observations on the capacity of the range of simple random walks
  on {$\mathbb Z^3$} and {$\mathbb Z^4$}.
\newblock {\em Electron. Commun. Probab.}, 22:Paper No. 25, 9, 2017.
  \MR{3652038}

\bibitem[Ein89]{einmahl1989extensions}
U.~Einmahl.
\newblock Extensions of results of {K}oml\'{o}s, {M}ajor, and {T}usn\'{a}dy to
  the multivariate case.
\newblock {\em J. Multivariate Anal.}, 28(1):20--68, 1989. \MR{996984}

\bibitem[Gra09]{gray09probability}
R.~M. Gray.
\newblock {\em Probability, random processes, and ergodic properties}.
\newblock Springer, Dordrecht, second edition, 2009. \MR{2840299}

\bibitem[HH24]{halberstam22logarithmic}
N.~Halberstam and T.~Hutchcroft.
\newblock Logarithmic corrections to the {A}lexander-{O}rbach conjecture for
  the four-dimensional uniform spanning tree.
\newblock {\em Comm. Math. Phys.}, 405(10):Paper No. 238, 45, 2024.
  \MR{4797749}

\bibitem[HLS24]{hernandez2024sharp}
S.~{Hernandez-Torres}, X.~{Li}, and D.~{Shiraishi}.
\newblock {Sharp one-point estimates and Minkowski content for the scaling
  limit of three-dimensional loop-erased random walk}.
\newblock {\em arXiv e-prints}, page arXiv:2403.07256, March 2024, 2403.07256.

\bibitem[HS23]{hutchcroft2020logarithmic}
T.~Hutchcroft and P.~Sousi.
\newblock Logarithmic corrections to scaling in the four-dimensional uniform
  spanning tree.
\newblock {\em Comm. Math. Phys.}, 401(2):2115--2191, 2023. \MR{4610293}

\bibitem[Hut20]{hutchcroft2020universality}
T.~Hutchcroft.
\newblock Universality of high-dimensional spanning forests and sandpiles.
\newblock {\em Probab. Theory Related Fields}, 176(1-2):533--597, 2020.
  \MR{4055195}

\bibitem[JO68]{jain68range}
N.~Jain and S.~Orey.
\newblock On the range of random walk.
\newblock {\em Israel J. Math.}, 6:373--380 (1969), 1968. \MR{243623}

\bibitem[Koz07]{kozma2007scaling}
G.~Kozma.
\newblock The scaling limit of loop-erased random walk in three dimensions.
\newblock {\em Acta Math.}, 199(1):29--152, 2007. \MR{2350070}

\bibitem[Law80]{lawler80selfavoiding}
G.~F. Lawler.
\newblock A self-avoiding random walk.
\newblock {\em Duke Math. J.}, 47(3):655--693, 1980. \MR{587173}

\bibitem[Law91]{lawler91intersections}
G.~F. Lawler.
\newblock {\em Intersections of random walks}.
\newblock Probability and its Applications. Birkh\"{a}user Boston, Inc.,
  Boston, MA, 1991. \MR{1117680}

\bibitem[LL10]{lawler2010random}
G.~F. Lawler and V.~Limic.
\newblock {\em Random walk: a modern introduction}, volume 123 of {\em
  Cambridge Studies in Advanced Mathematics}.
\newblock Cambridge University Press, Cambridge, 2010. \MR{2677157}

\bibitem[LS25a]{li2025convergence}
X.~Li and D.~Shiraishi.
\newblock Convergence of three-dimensional loop-erased random walk in the
  natural parametrization.
\newblock {\em Probab. Theory Related Fields}, 191(1-2):421--521, 2025.
  \MR{4869259}

\bibitem[LS25b]{losev2023long}
I.~Losev and S.~Smirnov.
\newblock How long are the arms in {DBM}?
\newblock {\em Comm. Math. Phys.}, 406(4):Paper No. 93, 16, 2025. \MR{4886061}

\bibitem[LSW04]{lawler2011conformal}
G.~F. Lawler, O.~Schramm, and W.~Werner.
\newblock Conformal invariance of planar loop-erased random walks and uniform
  spanning trees.
\newblock {\em Ann. Probab.}, 32(1B):939--995, 2004. \MR{2044671}

\bibitem[LSW19]{lawler19fourdimensional}
G.~Lawler, X.~Sun, and W.~Wu.
\newblock Four-dimensional loop-erased random walk.
\newblock {\em Ann. Probab.}, 47(6):3866--3910, 2019. \MR{4038044}

\bibitem[LV21]{lawler2021convergence}
G.~F. Lawler and F.~Viklund.
\newblock Convergence of loop-erased random walk in the natural
  parameterization.
\newblock {\em Duke Math. J.}, 170(10):2289--2370, 2021. \MR{4291423}

\bibitem[MNS21]{michaeli2021diameter}
P.~Michaeli, A.~Nachmias, and M.~Shalev.
\newblock The diameter of uniform spanning trees in high dimensions.
\newblock {\em Probab. Theory Related Fields}, 179(1-2):261--294, 2021.
  \MR{4221658}

\bibitem[MP10]{morters2010brownian}
P.~M\"{o}rters and Y.~Peres.
\newblock {\em Brownian motion}, volume~30 of {\em Cambridge Series in
  Statistical and Probabilistic Mathematics}.
\newblock Cambridge University Press, Cambridge, 2010.
\newblock With an appendix by Oded Schramm and Wendelin Werner. \MR{2604525}

\bibitem[PR04]{peres04scaling}
Y.~{Peres} and D.~{Revelle}.
\newblock {Scaling limits of the uniform spanning tree and loop-erased random
  walk on finite graphs}.
\newblock {\em arXiv Mathematics e-prints}, page math/0410430, October 2004,
  math/0410430.

\bibitem[Sch09]{schweinsberg2009loop}
J.~Schweinsberg.
\newblock The loop-erased random walk and the uniform spanning tree on the
  four-dimensional discrete torus.
\newblock {\em Probab. Theory Related Fields}, 144(3-4):319--370, 2009.
  \MR{2496437}

\bibitem[Shi18]{shiraishi2018growth}
D.~Shiraishi.
\newblock Growth exponent for loop-erased random walk in three dimensions.
\newblock {\em Ann. Probab.}, 46(2):687--774, 2018. \MR{3773373}

\bibitem[SS18]{sapozhnikov18brownian}
A.~Sapozhnikov and D.~Shiraishi.
\newblock On {B}rownian motion, simple paths, and loops.
\newblock {\em Probab. Theory Related Fields}, 172(3-4):615--662, 2018.
  \MR{3877544}

\bibitem[Whi02]{whitt2002stochastic}
W.~Whitt.
\newblock {\em Stochastic-process limits}.
\newblock Springer Series in Operations Research. Springer-Verlag, New York,
  2002.
\newblock An introduction to stochastic-process limits and their application to
  queues. \MR{1876437}

\end{thebibliography}

\end{document}